\documentclass{amsart}
\usepackage{graphicx}
\usepackage[arrow, matrix, curve]{xy}
\usepackage{array,amsmath,amssymb,amsthm,verbatim,epsfig}
\newtheorem{thm}{Theorem}[section]
\newtheorem{cor}[thm]{Corollary}
\newtheorem{lem}[thm]{Lemma}
\newtheorem{prop}[thm]{Proposition}
\theoremstyle{definition}
\newtheorem{defn}[thm]{Definition}
\theoremstyle{remark}
\newtheorem{rem}[thm]{Remark}
\newtheorem{ex}[thm]{Example}
\numberwithin{equation}{section}

\renewcommand{\O}{\mathcal{O}}
\newcommand{\K}{\mathcal{K}}
\newcommand{\G}{\mathcal{G}}
\newcommand{\idsh}{\mathcal{I}}
\newcommand{\ssh}{\mathcal{O}}
\newcommand{\N}{\mathbb{N}}
\newcommand{\Aff}{\mathbb{A}}
\newcommand{\suchThat}{\enspace \vert \enspace}

 \DeclareMathOperator{\rk}{rk}
 \DeclareMathOperator{\Hom}{Hom}

\DeclareMathOperator{\Spec}{Spec} \DeclareMathOperator{\Proj}{Proj}
 \DeclareMathOperator{\inv}{inv}

\DeclareMathOperator{\Tor}{Tor} 
\DeclareMathOperator{\Gl}{Gl} \DeclareMathOperator{\Sl}{Sl}
\DeclareMathOperator{\id}{id} \DeclareMathOperator{\W}{W}
\DeclareMathOperator{\Frob}{Frob}
\DeclareMathOperator{\Loop}{L}
\DeclareMathOperator{\X}{X}
\DeclareMathOperator{\T}{T}
\DeclareMathOperator{\weights}{P}
\DeclareMathOperator{\diag}{diag}

\DeclareMathOperator{\Grass}{Grass}

\newcommand{\cochar}{\check{\X}}
\newcommand{\domcochar}{\check{\X}_{+}}
\newcommand{\coweights}{\check{\weights}}
\newcommand{\fpv}[2]{F_{p^{#1}}(#2)}
\begin{document}

\title[Demazure resolutions and lattices with infinitesimal structure]{Demazure resolutions as varieties of lattices with infinitesimal structure}
\author{Martin Kreidl}
\address{Mathematisches Institut der Universit\"at Bonn, Beringstra\ss e 4, 53115 Bonn, Germany}
\email{kreidl@math.uni-bonn.de}



\begin{abstract}
Let $k$ be a field of positive characteristic. We construct, for each dominant coweight $\lambda$ of the standard maximal torus in $\Sl_{n}$, a closed subvariety $\mathcal{D}(\lambda)$ of the multigraded Hilbert scheme of an affine space over $k$, such that the $k$-valued points of $\mathcal{D}(\lambda)$ can be interpreted as lattices in $k((z))^{n}$ endowed with infinitesimal structure. Moreover, for any $\lambda$ we construct a universal homeomorphism from $\mathcal{D}(\lambda)$ to a Demazure resolution of the Schubert variety $\mathcal{S}(\lambda)$ associated with $\lambda$ in the affine Grassmannian.
Lattices in $\mathcal{D}(\lambda)$ have non-trivial infinitesimal structure if and only if they lie over the boundary of the big cell of $\mathcal{S}(\lambda)$.
\end{abstract}
\maketitle

\section{Introduction}

Let $k$ be a field and consider the group $G=\Sl_{n}$ over $k$. Associated with $G$ we have the affine Grassmannian $\mathcal{G}$ which we regard as a functor from $k$-algebras to sets, whose set of $k$-valued points is the set of lattices $\mathcal{L} \subset k((z))^{n}$ such that $\wedge^{n}\mathcal{L} = k[[z]]$. The affine Grassmannian is the inductive limit of the $k$-schemes $\mathcal{G}^{(N)}$, where $\mathcal{G}^{(N)}(k)=\lbrace \mathcal{L} \in \mathcal{G} \text{ such that } z^{N}k[[z]]^{n}\subset\mathcal{L}\subset z^{-N}k[[z]]^{n}\rbrace$. Each `finite piece' $\mathcal{G}^{(N)}$ is seen to be a closed subscheme of the ordinary Grassmannian $\Grass_{nN,2nN}$ by identifying a lattice $\mathcal{L}$ with its image in $z^{-N}k[[z]]^{n}/z^{N}k[[z]]^{n}\simeq k^{2nN}$. From this point of view, $\mathcal{G}^{(N)}$ is just the closed subvariety of $\Grass_{nN,2nN}$ which parametrizes those $nN$-dimensional subspaces of $k^{2nN}$ which are stable under the endomorphism of $k^{2nN}$ corresponding to multiplication by $z$ in $k((z))^{n}$. Equivalently, one could say that $\mathcal{G}^{(N)}$ parametrizes $nN$-dimensional linear subvarieties of $\Aff^{2nN}_{k}$ which are stable under the endomorphism of $\Aff_{k}^{2nN}$ `given by multiplication by $z$'.
Moreover, $\mathcal{G}$ carries a natural $\Sl_{n}(k[[z]])$-operation whose orbit-closures are by definition the Schubert-varieties in $\mathcal{G}$. Any Schubert variety $\mathcal{S}$ can be realized as a closed subvariety of a finite piece $\mathcal{G}^{(N)}$ of the affine Grassmannian.

In this paper we are going to investigate a variant of the `finite pieces' of the affine Grassmannian by considering parameter spaces for certain \emph{non-linear} subvarieties of $\Aff_{k}^{nN}$. This will lead us to objects which are closely related to Demazure varieties.
\smallskip

From now on, let $k$ have positive characteristic $p$.
\smallskip

\noindent For any $k$-algebra $R$, there is a ring structure $R[[z]]^{F}$ on the set $R^{\mathbb{N}}$, with addition given componentwise and multiplication defined by
\begin{multline*}
(a_{0}+a_{1}z+a_{2}z^{2}+\dotsb)(b_{0}+b_{1}z+b_{2}z^{2}+\dotsb) := \\ := a_{0}b_{0}+(a_{0}^{p}b_{1}+a_{1}b_{0}^{p})z+(a_{0}^{p^{2}}+a_{1}^{p}b_{1}^{p}+a_{2}b_{0}^{p^{2}})z^{2}+\dotsc.
\end{multline*}
We call this ring the `ring of Frobenius-twisted power series'.

Let us consider $\Aff^{nN}_{k} = (\Aff^{N}_{k})^{n}$ as a scheme endowed with the module structure defined by the ring operations in $k[[z]]^{F}$. On its coordinate ring $k[x_{i,j}; i=1,\dots,n; j=0,\dotsc,N-1]$ we fix the grading given by $\deg x_{i,j}=p^{j}$. By a result of Haiman and Sturmfels, \cite{haiman-2002}, there exists for any noetherian graded $k$-algebra $A$ a $k$-scheme, called the `multigraded Hilbert scheme', parametrizing `admissible' ideals $I\subset A$ ($I$ is called admissible if $I$ is graded with constant Hilbert function and locally free quotient $A/I$). In particular, there exists a multigraded Hilbert scheme for $\Aff_{k}^{nN}$ with the above grading, and it carries a natural action of $\Sl_{n}(k[[z]]^{F})$. Let us call a closed subscheme $V\subset \Aff_{k}^{nN}$ defined by an admissible ideal a \emph{lattice scheme}, if $V$ is stable under the module operations defined on $\Aff_{k}^{nN}$. Then the action of $\Sl_{n}(k[[z]]^{F})$ on the multigraded Hilbert scheme restricts to an action on the \emph{closed subvariety} of lattice schemes.
With any dominant (with respect to the standard Borel) coweight $\lambda$ of the standard maximal torus $T\subset\Sl_{n}$ we associate the lattice scheme $V(\lambda)\subset \Aff_{k}^{nN}$ defined by
$$
I(\lambda)=\langle x_{1,0},\dotsc,x_{1,\lambda_{1}-\lambda_{n}},\dotsc,x_{n-1,0},\dotsc,x_{n-1,\lambda_{n-1}-\lambda_{n}} \rangle.
$$
We think of $V(\lambda)$ as the `standard lattice' with elementary divisors $z^{\lambda_{1}},\dotsc,z^{\lambda_{n}}$. Let us denote by $\mathcal{D}(\lambda)$ its orbit closure in the multigraded Hilbert scheme.

Furthermore, for any such $\lambda\in \domcochar(T)\simeq \mathbb{Z}^{n}$ we define a sequence $\mu_{1},\dotsc,\mu_{N}\in \mathbb{Z}^{n}$ by setting
$\mu_{j} := (1,\dotsc,1,0,\dotsc,0)$
such that the number of 1's equals the number of indices $1\leq i\leq d$ with $\lambda_{i}-\lambda_{n}\geq j$. In particular, we have $\sum_{j=1}^{N}\mu_{i} = \lambda - \lambda_{n}(1,\dotsc,1)$. This sequence of $\mu_{j}$'s defines a sequence of minuscule dominant coweights of the standard maximal torus in $\mathbb{P}\Gl_{n}$ (with respect to the standard Borel), and hence determines a Demazure resolution $\pi(\lambda): \Sigma(\mu_{1},\dotsc,\mu_{N})\to \mathcal{S}(\lambda)$ of the Schubert variety $\mathcal{S}(\lambda)\subset \mathcal{G}$ associated with $\lambda$.

The varieties $\mathcal{D}(\lambda)$ and $\Sigma(\mu_{1},\dotsc,\mu_{N})$ are related according to the following

\begin{thm}\label{thmIntro}
The $k$-variety $\mathcal{D}(\lambda)$ is an iterated bundle of ordinary Grassmannians, and there is a universal homeomorphism $$\sigma: \mathcal{D}(\lambda)\to \Sigma(\mu_{1},\dotsc,\mu_{N}).$$ Furthermore, let
$\pi'(\lambda): \mathcal{D}(\lambda)\to \Sigma(\mu_{1},\dotsc,\mu_{N}) \to \mathcal{S}(\lambda)$
denote the composition of $\sigma$ with the Demazure resolution $\pi(\lambda)$ of $\mathcal{S}(\lambda)$. Then a lattice scheme given by a $k$-valued point in $\mathcal{D}(\lambda)$ is reduced if and only if it maps to the big cell of $\mathcal{S}(\lambda)$ under $\pi'(\lambda)$.
\end{thm}

Put otherwise: the fiber of $\mathcal{D}(\lambda)$ over a $k$-valued point $P$ in $\mathcal{S}(\lambda)$ has positive dimension if and only if $P$ lies in the boundary of the big cell. The points in this fiber correspond to different infinitesimal structures on the lattice corresponding to $P$.
\medskip

There is an easy way to conceptually understand the ring of Frobenius-twisted power series: For a ring $R$ of characteristic $p>0$, identify $\W(R)$, the ring of Witt vectors over $R$, and $R^{\mathbb{N}}$ as sets and let $I_{n}\subset \W(R)$ be the ideal with underlying set $\lbrace 0 \rbrace^{n}\times R^{\mathbb{N}} \subset R^{\mathbb{N}}$. The ideals $I_{n}$ define a filtration $\mathcal{I}: \W(R)=I_{0}\supset I_{1}\supset I_{2}\supset \dotsb$ whose associated graded ring gives, after completion with respect to the topology defined by $\mathcal{I}$, the ring of Frobenius-twisted power series $R[[z]]^{F}$.
In fact, the study of spaces of lattices over $\W(k)$ was the original motivation to consider lattices over $k[[z]]^{F}$: Also in the `Witt vector case' one can carry out the construction of a `standard lattice scheme' $V(\lambda)$ with given elementary divisors and its orbit closure $\mathcal{D}_{\W}(\lambda)$ under $\Sl_{n}(\W(k))$ in the multigraded Hilbert scheme. This orbit closure could be viewed as an analogon of Schubert varieties in the Witt vector case. It turns out that also here the phenomenon of infinitesimal structures on lattice schemes arises, and that there are in general many different lattice schemes in the orbit closure of $V(\lambda)$ giving rise to the same reduced structure. In the special case $n=2,\lambda=(1,-1)$ and $p\geq 3$ simple calculations show that the varieties $\mathcal{D}(\lambda)$ and $\mathcal{D}_{\W}(\lambda)$ are equal (we calculate $\mathcal{D}(\lambda)$ in section \ref{sectionIsomorphismTheorem} - the calculation in the Witt vector case is completely parallel). Unfortunately, the general results in the case of Frobenius-twisted power series cannot be carried over directly to the Witt vector case, since in the latter case the ideals which define lattice schemes fail to be `twisted-linear', which is essential for our results (see below). We will come back to the case of Witt vectors in a subsequent paper.

We mention that there is Haboush's paper \cite{haboush} in which he studies spaces of lattices over the ring of Witt vectors - however, in this paper he does not account for the phenomenon of infinitesimal structures.
\smallskip

We will start this paper by recalling in section \ref{sectionDemazureVarieties} the notions of affine Grassmannian, Schubert varieties and their Demazure resolutions. The main purpose of this section is to clarify the connection between the most common description of a Demazure variety in terms of quotients of products of parahorics, and an equivalent description as a variety of lattice chains. In this paper we will work with the latter description.
In section \ref{sectionPowerSeriesRings} we introduce the ring of Frobenius twisted power series. This ring structure gives rise to a structure of ring scheme on the affine space $\Aff_{k}^{N}$, and to a structure of ($n$-dimensional) $\Aff_{k}^{N}$-module scheme on $\Aff_{k}^{nN}$. Flat subschemes of $\Aff_{k}^{nN}$, which are stable with respect to the addition- and scalar multiplication-morphisms, will be called lattice schemes, and are parametrized by a closed subscheme $\mathcal{M}$ of the multigraded Hilbert scheme. This is shown in section \ref{sectionLatticeSchemes}.
The technical heart of the paper is section \ref{sectionTechnical}. Here we introduce the notion of `twisted-linear` ideals in a polynomial ring of positive characteristic, and prove functoriality and flatness results for twisted linear ideals. Moreover we prove that twisted linearity is a `closed condition` on the multigraded Hilbert scheme.
Finally, the last section is devoted to the construction of an isomorphism from $\mathcal{D}(\lambda)$ to an iterated bundle of ordinary Grassmannians (theorem \ref{thmGrassmannBundle} and corollary \ref{corMain}) and to the construction of the universal homeomorphism $\sigma: \mathcal{D}(\lambda)\to \Sigma(\mu_{1},\dotsc,\mu_{N})$ (theorem \ref{thmRelFunctors}). As an easy corollary (corollary \ref{corInfStruct}) we obtain that a lattice scheme $V\in\mathcal{D}(\lambda)(k)$ maps to the big cell of $\mathcal{S}(\lambda)$ if and only if $V$ is reduced, which proves the last assertion of theorem \ref{thmIntro}.
Also we illustrate these results by explicitly calculating the respective objects and morphisms for the `smallest` non-trivial example, given by $n=2,N=2, \lambda=(1,-1)$.
\smallskip

\textbf{Acknowledgement.}
I am very grateful to Ulrich G\"ortz for his encouragement and constant interest in this work and for pointing out serious errors in a preliminary version as well as for making a number of other valuable suggestions on this paper. Also I want to thank Michael Rapoport and his group for helpful discussions and the opportunity to work in such a stimulating environment. This work was partially supported by the SFB/TR 45 `Periods, Moduli Spaces and Arithmetic of Algebraic Varieties' of the DFG (German Research Foundation).


\section{The affine Grassmannian and Demazure varieties}\label{sectionDemazureVarieties}

In this section we describe the construction of the affine Grassmannian associated to the special linear group $\Sl_{n}$, as well as the Demazure-Hansen-Bott-Samelson desingularization of its Schubert cells. For both objects we recall the `standard`-definition in terms of quotients of loop groups, and give (also well known) equivalent descriptions as a variety of lattices, and a variety of lattice chains, respectively. These latter descriptions are those which we are going to use in the rest of the paper.

Throughout this section, we denote by $k$ an arbitrary algebraically closed field, and write $\O = k[[z]]$ for the ring of power series in one variable, and $\K=k((z))$ for its quotient field.

\subsection{The affine Grassmannian for $\Sl_{n}$}\label{subsectionGrassmannian} The material of this section can be found in more detail in the article \cite{beauville-laszlo} of Beauville and Laszlo.

Let $G=\Sl_{n}$, let $T\subset B \subset G$ be the standard maximal torus (of diagonal matrices) and the standard Borel subgroup (of upper triangular matrices) of $G$. Moreover, denote by $\cochar(T)$ the group of cocharacters of $T$, and by $\domcochar(T)$ the subset of dominant cocharacters (respective to the fixed Borel $B$).

Consider the following (set)-valued functors on the category of $k$-algebras:
\begin{align}
\Loop^{\geq 0}G: R \mapsto G(R[[z]]) \text{ (the `positive loop group` associated to $G$)},\\
\Loop G: R\mapsto G(R((z))) \text{ (the `loop group` associated to $G$)}.
\end{align}
It is easy to see that the positive loop group is represented by an (infinite dimensional) affine $k$-scheme, while the loop group is represented by an ind-scheme over $k$ (i.e. an inductive limit of $k$-schemes in the category of $k$-spaces). In both cases, we will not further distinguish between the (ind-)schemes and the functors which they represent.

We identify $\domcochar(T)$ with a subset of $\mathbb{Z}^{n}$ and hence obtain the injection
\begin{equation}\label{eqnEvalMap}
\cochar(T) \to \Loop G; \lambda \mapsto z^{\lambda} = \diag(z^{\lambda_{1}},\dotsc,z^{\lambda_{n}}).
\end{equation}
Thus we can regard the group of cocharacters of $T$ with a subset of $\Loop G(R)$ for any $R$.

\begin{defn}\label{defAffineGrassmannian}
The quotient $\Loop G/\Loop^{\geq 0}G$ in the category of $k$-spaces, associated to the functor
\begin{equation}
 R\mapsto G(R((z)))/G(R[[z]])
\end{equation}
is called the \emph{affine Grassmannian} (associated to $G$). In the sequel we denote the affine Grassmannian by $\mathcal{G}$.
\end{defn}

On the $k$-space $\G$ we have a natural action of $\Loop^{\geq 0}G$ by multiplication on the left. On the level of $k$-valued points, this action induces a double coset decomposition for $\Loop G(k)$, namely
\begin{equation}\label{eqnDoubleCosetDec}
\Loop G(k) = \displaystyle\cup_{\lambda\in\domcochar(T)}\Loop^{\geq 0}G(k)z^{\lambda}\Loop^{\geq 0}G(k).
\end{equation}
This shows that the $\Loop^{\geq 0}G(k)$-orbits in $\G(k)$ are parametrized by the dominant cocharacters of $T$. These orbits are called the \emph{Schubert cells} of $\G$, and denoted in the sequel by $\mathcal{C}(\lambda)$ for $\lambda\in\domcochar(T)$.

It is known that $\G$ is in fact represented by an ind-scheme. Before we state the precise result, let us recall the notion of lattice, which is essential for the proof (and the rest of this paper).

\begin{defn}\label{defLattice}
An $R[[z]]$-submodule $\mathcal{L}\subset R((z))^{n}$ is called a \emph{lattice}, if
\begin{enumerate}
\item there exists $N\in \mathbb{N}$, such that $z^{N}R[[z]]^{n}\subset \mathcal{L}\subset z^{-N}R[[z]]$, and
\item the quotient $z^{-N}R[[z]]/\mathcal{L}$ is a projective $R$-module.
\end{enumerate}
The lattice $\mathcal{L}$ is called \emph{special}, if $\wedge^{n}\mathcal{L} = R[[z]]$, i.e. $\mathcal{L}$ has fpqc-locally a basis with determinant 1.
\end{defn}

\begin{thm}[Beauville, Laszlo, \cite{beauville-laszlo}]\label{thmBeauvilleLaszlo}
The affine Grassmannian $\G$ is isomorphic to the functor, which associates to every $k$-algebra $R$ the set of special lattices in $R((z))^{n}$. Moreover, it is represented by an inductive limit of closed subschemes of usual Grassmannians, i.e. it is an ind-scheme and even ind-projective.
\end{thm}

The idea of the proof is to write $\G = \cup_{N\in\mathbb{N}}\G^{(N)}$, where $$\G^{(N)}(R)= \lbrace \text{ special lattices }\mathcal{L}\text{ such that } z^{N}R[[z]]^{n}\subset \mathcal{L}\subset z^{-N}R[[z]]^{n}\text{ }\rbrace.$$ The latter are shown to be closed subschemes of usual Grassmannians.

It is easy to see that $\Loop^{\geq 0}G$ acts algebraically on every $\G^{(N)}$, whence the Schubert cells are quasi-projective $k$-schemes, each lying in a suitable $\G^{(N)}$. We write $\mathcal{S}(\lambda)$ for the closure of a Schubert cell $\mathcal{C}(\lambda)$ and call it a \emph{Schubert variety} in $\G$.

\begin{rem}
For $G=\Sl_{n}$ we have defined the affine Grassmannian as an fpqc-quotient $\Loop G/\Loop^{\geq 0}G$. However, the same definition applies to other reductive groups $G$; let us consider the case $G=\Gl_{n}$. We denote by $\mathcal{G}_{\Gl_{n}}$ the \emph{affine Grassmannian for $\Gl_{n}$}, i.e. the quotient $\Loop \Gl_{n}/\Loop^{\geq 0}\Gl_{n}$ in the category of $k$-spaces, associated to the functor $R\mapsto \Gl_{n}(R((z)))/\Gl_{n}(R[[z]])$. Then, in the same way as for $\Sl_{n}$, we obtain (\cite{beauville-laszlo}): The affine Grassmannian $\mathcal{G}_{\Gl_{n}}$ is isomorphic to the functor, which associates to every $k$-algebra $R$ the set of lattices in $R((z))^{n}$. Moreover, $\mathcal{G}_{\Gl_{n}}$ is an ind-scheme over $k$.
\end{rem}

\subsection{Iwahori and parahoric subgroups}

We return to our notation $G=\Sl_{n}$. Let $\epsilon: G(\O)=G(k[[z]])\to G(k)$ be the map induced by $z\mapsto 0$. The \emph{standard Iwahori subgroup} $\mathcal{I} \subset G(\O)$ is by definition the preimage of the standard Borel subgroup $B\subset G$. Any subgroup of $G(\K)$ containing $\mathcal{I}$ is called a \emph{standard parahoric subgroup} of $G(\K)$. In general, an \emph{Iwahori-(resp. parahoric) subgroup} of $G(\K)$ is a $G(\K)$-conjugate of $\mathcal{I}$ (resp. a standard parahoric subgroup). Note that Iwahori- (resp. parahoric) subgroups are exactly the stabilizers of complete (resp. partial) lattice chains
$$
z\mathcal{L}\subset\mathcal{L}_{1}\subset\dotsb\subset\mathcal{L}_{l}\subset\mathcal{L},\quad l\leq n-1,
$$
of lattices. In particular, the standard Iwahori subgroup is the stabilizer of the \emph{standard lattice chain}
$$
zk[[z]]^{n}\subset k[[z]]\oplus zk[[z]]^{n-1}\subset\dotsb\subset k[[z]]^{n-1}\oplus zk[[z]]\subset k[[z]]^{n},
$$
while the standard parahoric subgroups are the stabilizers of its respective subflags. A \emph{maximal} standard parahoric subgroup is the stabilizer of a single lattice in the standard lattice chain.

Let $\coweights_{+}$ be the set of dominant coweights of $\mathbb{P}\Gl_{n}$ w.r.t. the standard maximal torus and the standard Borel of upper trinangular matrices, and let $\coweights_{+,\min}\subset \coweights_{+}$ be the subset of minuscule dominant coweights. (A system of representatives of $\coweights_{+,\min}$ is given by the vectors of the form $(1,\dotsc,1,0,\dotsc,0)$, with the number of 1's running from $0$ to $d-1$.) In the sequel, if we speak of minuscule dominant coweights, we will always be referring to elements of $\coweights_{+,\min}$. By identifying $\coweights_{+} = \mathbb{Z}^{n}/(1,\dotsc,1)\mathbb{Z}$ we can write any $\lambda \in \coweights_{+}$ as $\lambda=(\lambda_{1},\dotsc,\lambda_{n})$ and note that the number $\lvert \lambda \rvert = \sum \lambda_{i}$ is well defined modulo $n$.
We note the bijection
\begin{equation}\label{eqnConjCoweightParahoric}
\begin{split}
\coweights_{+,\min} &\to \lbrace \text{ maximal standard parahoric subgroups }\rbrace\\
\mu &\mapsto P_{\mu}:=z^{-\tilde{\mu}}G(\O)z^{\tilde{\mu}},
\end{split}
\end{equation}
where $\tilde{\mu}\in \mathbb{Z}^{n}$ is a representative of $\mu$. To abbreviate notation in the following section, we set $P_{\mu_{1},\mu_{2}} := P_{\mu_{1}} \cap P_{\mu_{2}}$, the stabilizer of the two respective lattices in the standard flag, for $\mu_{i}$ minuscule dominant coweights.

\subsection{Demazure-Hansen-Bott-Samelson desingularization of Schubert varieties}\label{subsectionDemazure}

The presentation of the material in this section follows roughly the exposition by Gaussent and Littelmann in \cite{gaussent-2003}, and complements it with an alternative description of Demazure-Hansen-Bott-Samelson varieties in terms of lattice chains. The results are essentially due to Contou-Carr\`ere \cite{contoucarrere-1983}. Note that instead of `Demazure-Hansen-Bott-Samelson variety` we will often simply speak of `Demazure variety`. This is common in the literature.

\begin{defn}[Demazure variety]\label{defBottSamelson}
Let $\nu_{1},\dotsc,\nu_{m}$ be a sequence of minuscule dominant coweights. The \emph{Demazure-Hansen-Bott-Samelson variety} $\Sigma(\nu_{1},\dotsc,\nu_{m})$ is defined as
\begin{equation*}
\Sigma(\nu_{1},\dotsc,\nu_{m}) = P_{0} \times^{P_{0, \nu_{1}}} P_{\nu_{1}} \times^{P_{\nu_{1}, \nu_{2}}} \dotsb \times^{P_{\nu_{m-1}, \nu_{m}}} P_{\nu_{m}}/P_{\nu_{m},0},
\end{equation*}
i.e. the variety $P_{0} \times P_{\nu_{1}} \times \dotsb \times P_{\nu_{m}}$ modulo the right-action of the subgroup $P_{0, \nu_{1}} \times \dotsb \times P_{\nu_{m-1}, \nu_{m}} \times P_{\nu_{m},0}$ given by
$$
(g_{0},\dotsc,g_{m})\cdot (q_{0},\dotsc,q_{m}) = (g_{0}q_{0},q_{0}^{-1}g_{1}q_{1},\dotsc,q_{m-1}^{-1}g_{m}q_{m}).
$$
\end{defn}

There is a natural morphism
\begin{equation}\label{eqnResolution}
\begin{split}
\pi: \Sigma(\nu_{1},\dotsc,\nu_{m}) &\to \G\\ [g_{0},\dotsc,g_{m}] &\mapsto g_{0}\dotsb g_{m}G(\O)/G(\O),
\end{split}
\end{equation}
whose image is a closed $G$-invariant subvariety of $\G$.

Let us remark here that in \cite{gaussent-2003} Gaussent and Littelmann define a more general notion of Demazure variety in terms of sequences of `types` (admitting not only maximal, but any standard parahoric subgroups in the above definition). However, in our situation, to give a sequence of types in the sense of \cite{gaussent-2003} is the same as to give a sequence of minuscule dominant coweights.

We are now going to describe Demazure-Hansen-Bott-Samelson varieties in terms of (descending) lattice chains: With any sequence $\mu_{1},\dotsc,\mu_{m+1}\in \lbrace 0,1\rbrace^{n}$, such that every $\mu_{i}$ is of the form $(1,\dotsc,1,0,\dotsc,0)$ (i.e. represents a minuscule dominant coweight) and $\sum\mu_{i} \equiv 0 \mod n$, we associate a variety of lattice chains $\tilde{\Sigma}(\mu_{1},\dotsc,\mu_{m+1})\subset \prod_{j=1}^{m+1}\mathcal{G}_{\Gl_{n}}$, defined on the level of $k$-valued points by
\begin{multline}\label{eqnLatticeChains}
\tilde{\Sigma}(\mu_{1},\dotsc,\mu_{m+1})(k) = \lbrace (\mathcal{L}_{0}=k[[z]]^{n},\mathcal{L}_{1}, \dotsc, \mathcal{L}_{m+1}) \suchThat \inv(\mathcal{L}_{i},\mathcal{L}_{i+1}) = \mu_{i+1} \rbrace.
\end{multline}
Here, by $\inv(\mathcal{L}_{i},\mathcal{L}_{i+1})$ we denote the vector of elementary divisors of $\mathcal{L}_{i+1}$ relative to $\mathcal{L}_{i}$, ordered by decreasing size. Note that $\tilde{\Sigma}(\mu_{1},\dotsc,\mu_{m+1})(k)$ is a subset of the set of $k$-valued points of a product of affine Grassmannians for $\Gl_{n}$. Thus we obtain a well-defined reduced scheme structure on $\tilde{\Sigma}(\mu_{1},\dotsc,\mu_{m+1})$.

\begin{rem}\label{remBuilding}
This definition implies that for any $i$ we have $z\mathcal{L}_{i}\subset \mathcal{L}_{i+1}\subset \mathcal{L}_{i}$. Thus the points of $\tilde{\Sigma}(\mu_{1},\dotsc,\mu_{m+1})$ correspond to sequences of vertices in the affine building of $\Sl_{n}$, where two subsequent vertices are joint by a 1-dimensional face in the building.
\end{rem}

As for $\Sigma(\nu_{1},\dotsc,\nu_{m})$, there is a morphism
\begin{equation}\label{eqnAltResolution}
\begin{split}
\tilde{\pi}: \tilde{\Sigma}(\mu_{1},\dotsc,\mu_{m+1}) &\to \G\\
    (\mathcal{L}_{0}, \dotsc, \mathcal{L}_{m+1}) &\mapsto \mathcal{L}_{m+1}z^{-\sum\mu_{i}}.
\end{split}
\end{equation}
(One has to check that $\mathcal{L}_{m+1}z^{-\sum\mu_{i}}$ is special, but this is easy, as it is sufficient to consider only $k$-valued points.)

Now, for each $i$, let $\nu_{i}$ be as in the definition of Demazure variety, and assume $\mu_{i}$ chosen to be a representative of the unique minuscule dominant coweight lying in the $W$-orbit of $\nu_{i-1}-\nu_{i}$ ($W=\operatorname{S}_{n}$ the Weyl group of $\mathbb{P}\Gl_{n}$) (Of course, this can be reversed: to any sequence of $\mu_{i}$'s we can find the corresponding $\nu_{i}$'s). Moreover, set $\nu_{m+1}=0$, and for every $i=1,\dotsc,m+1$ choose a representative $\tilde{\nu}_{i}$ of $\nu_{i}$ such that $\tilde{\nu}_{i-1}-\tilde{\nu}_{i}\in \lbrace 0,1\rbrace^{n}$. Then we have

\begin{prop}\label{propLatticeChains}
Let $[g_{0},\dotsc,g_{m}]\in\Sigma(\nu_{1},\dotsc,\nu_{m})$. The assignment 
\begin{equation}\label{eqnLattice}
 \mathcal{L}_{i}=g_{0}\dotsb g_{i-1}z^{-\tilde{\nu}_{i}}\mathcal{L}_{0}, \quad i=1,\dotsc,m+1,
\end{equation}
defines an isomorphism of varieties $$\varphi: \Sigma(\nu_{1},\dotsc,\nu_{m}) \xrightarrow{\simeq} \tilde{\Sigma}(\mu_{1},\dotsc,\mu_{m+1}).$$ Furthermore, we have $\varphi\circ\tilde{\pi} = \pi$, and the morphisms \eqref{eqnResolution} and \eqref{eqnAltResolution} are desingularitizations of the Schubert variety $\mathcal{S}(\sum_{i=1}^{m+1}\mu_{i} \mod n)$.
\end{prop}

\begin{proof}
Let $(g_{0},\dotsc,g_{m})$ be a representative of an $R$-valued point in $\Sigma(\nu_{1},\dotsc,\nu_{m})$ ($R$ any $k$-algebra) and let
$\mathcal{L}_{i}$ be as in the statement of the proposition, i.e. the lattice generated by the columns of the matrix $g_{0}\dotsb g_{i-1}z^{-\tilde{\nu}_{i}}$ for $i=1,\dotsc,m+1$.
We calculate
\begin{multline*}
\inv(\mathcal{L}_{i},\mathcal{L}_{i+1}) = \inv(\mathcal{L}_{0},z^{\tilde{\nu}_{i}}g_{i}z^{-\tilde{\nu}_{i+1}}\mathcal{L}_{0}) = \\ = \inv(\mathcal{L}_{0},hz^{\tilde{\nu}_{i}-\tilde{\nu}_{i+1}}\mathcal{L}_{0})\text{   (for some $h\in G(\O)$)  }= \mu_{i+1},
\end{multline*}
and thus \eqref{eqnLattice} gives indeed a point in $\tilde{\Sigma}(\mu_{1},\dotsc,\mu_{m+1})$. It is independent of the representative $(g_{0},\dotsc,g_{m})$, and indeed $\varphi\circ\tilde{\pi} = \pi$. We are left with the construction of an inverse map, which we will first explain on the level of $k$-valued points, i.e. lattices in $k((z))^{n}$.
Let $(\mathcal{L}_{0}, \dotsc, \mathcal{L}_{m+1}) \in \tilde{\Sigma}(\mu_{1},\dotsc,\mu_{m+1})$, and assume we have constructed $g_{0}\in P_{\nu_{0}},\dotsc,g_{i-1}\in P_{\nu_{i-1}}$ up to the right action of
$P_{0,\nu_{1}}\times\dotsb\times P_{\nu_{i-1},\nu_{i}}$, such that
$\mathcal{L}_{j}=g_{0}\dotsb g_{j-1}z^{-\tilde{\nu}_{j}}\mathcal{L}_{0}$ for $j\leq i$.
Let $\mathcal{L}_{i+1}=h_{i}\mathcal{L}_{0}$, and set $g_{i}= g_{i-1}^{-1}\dotsb g_{0}^{-1}h_{i} z^{\tilde{\nu}_{i+1}}$. Then the equation
$$
\inv(\mathcal{L}_{0},z^{\tilde{\nu}_{i}}g_{i}z^{-\tilde{\nu}_{i+1}}\mathcal{L}_{0}) = \inv(\mathcal{L}_{i},\mathcal{L}_{i+1}) = \mu_{i+1}
$$
shows that $g_{i}\in P_{\nu_{i}}P_{\nu_{i+1}}$. Moreover, for an appropriate choice of representative $h_{i}$ we even get $g_{i}\in P_{\nu_{i}}$, which determines $g_{i}$ up to right action by $P_{\nu_{i},\nu_{i+1}}$. Hence, by induction, we get a unique preimage for any sequence of lattices in $\tilde{\Sigma}(\mu_{1},\dotsc,\mu_{m+1})$. In order to obtain a true morphism, we have to describe the map on the level of $R$-valued points for any local $k$-algebra $R$. However, the quotient $\Loop \Gl_{n}\to \Loop \Gl_{n}/\Loop^{\geq 0}\Gl_{n}$ is locally trivial for the Zariski-topology (see e.g. Faltings \cite{faltings-2003}), which shows that the above construction can indeed be carried out for $R$-valued points, $R$ any local $k$-algebra. This proves the first claim.
Finally note that by remark \ref{remBuilding} $\tilde{\Sigma}(\mu_{1},\dotsc,\mu_{m+1})$ is a twisted product of ordinary Grassmannians and thus smooth and projective. Since the dimensions of source and target of the morphisms in question are equal and the involved schemes are projective over $k$, the second claim follows.
\end{proof}

Note that the dominant cocharacters of $\Sl_{n}$ inject into those of $\mathbb{P}\Gl_{n}$. Thus $\sum_{i=1}^{m+1}\mu_{i}$ determines uniquely a dominant cocharacter of $\Sl_{n}$ and hence the Schubert variety $\mathcal{S}(\sum_{i=1}^{m+1}\mu_{i} \mod n)$. Conversely, any dominant cocharacter of $\Sl_{n}$ can be decomposed into a sequence of minuscule dominant coweights, which determines sequences $\mu_{i}$ and $\nu_{i}$ as above.


For later use we state explicitly a well known formula for the dimension of $\Sigma(\nu_{1},\dotsc,\nu_{m})$ (see e.g. \cite{gaussent-2003}): Let $\rho$ be half the sum of the positive roots $\lbrace\alpha_{i,j}\rbrace$ of $Sl_{n}$, and let $\lambda$ be the dominant cocharacter of $\Sl_{n}$ given by $\sum_{i=1}^{m+1}\mu_{i} \mod n$. Then
\begin{equation}
 \dim \tilde{\Sigma}(\mu_{1},\dotsc,\mu_{m+1}) = \dim \Sigma(\nu_{1},\dotsc,\nu_{m}) = 2\langle \lambda,\rho \rangle.
\end{equation}
As a sketch of proof, we just note the following line of equations:
\begin{multline*}
 2\langle \lambda,\rho \rangle = \sum_{j>i}\langle \lambda,\alpha_{i,j}\rangle = \sum_{i=1}^{m+1}\lvert\mu_{i}\rvert(n-\lvert\mu_{i}\rvert) = \\ = \sum_{i=1}^{m+1}\dim \Grass_{\lvert\mu_{i}\rvert,n} = \dim \tilde{\Sigma}(\mu_{1},\dotsc,\mu_{m+1}),
\end{multline*}
where $\Grass_{\lvert\mu_{i}\rvert,n}$ denotes the ordinary (finite) Grassmannian of $\lvert\mu_{i}\rvert$-dimensional subspaces in $k^{n}$.


\section{Frobenius-twisted power series rings}\label{sectionPowerSeriesRings}

For this and the following sections let $k$ denote a perfect field of positive characteristic $p$, and let $R$ be a $k$-algebra.
We have seen in section \ref{sectionDemazureVarieties} that the $R$-valued points of the affine Grassmannian (for $\Sl_{n}$) are just the special lattices in $R((z))^{n}$. One can view lattices lying between $z^{N}R[[z]]$ and $z^{-N}R[[z]]$ (and this is indeed the key idea in the proof of theorem \ref{thmBeauvilleLaszlo}) as linear subspaces of $\Aff_{R}^{2nN}$, which are stable under the endomorphism $\Aff_{R}^{2nN}\to\Aff_{R}^{2nN}$ induced by multiplication by $z$ on $R((z))^{n}$.

We are now going to define a modified ring structure on the set $R[[z]]$ of power series in one indeterminate and study subvarieties of affine space over $R$, which are stable with respect to this new ring structure. Finally, we will show in section \ref{sectionLatticeSchemes} that these subvarieties are parametrized by a closed subscheme of the multigraded Hilbert scheme.


Let $\W(R)$ be the ring of Witt vectors over $R$ and identify it as a set with $R^{\mathbb{N}}$. On $\W(R)$ we consider the filtration $\mathcal{I}: \W(R)=I_{0}\supset I_{1}\supset I_{2}\supset \dotsb$, where, for every $n\in\mathbb{N}$, $I_{n}$ is the ideal with underlying set $\lbrace 0 \rbrace^{n}\times R^{\mathbb{N}} \subset R^{\mathbb{N}}$.

\begin{defn}
(1) The \emph{ring of Frobenius-twisted power series} over $R$ is the completion of the associated graded ring $\operatorname{gr}_{\mathcal{I}}\W(R)$ with respect to the filtration given by the ideals $\oplus_{i\geq N}I_{i}/I_{i+1}$. We denote this ring by $R[[z]]^{F}$.\\
(2) The \emph{ring of (truncated) Frobenius-twisted power series of length $N$} is the quotient of $\operatorname{gr}_{\mathcal{I}}\W(R)$ by the ideal $\oplus_{i\geq N}I_{i}/I_{i+1}$. We denote it by $R[[z]]^{F}_{N}$.
\end{defn}

Note that the ring $R[[z]]^{F}$ contains $R = \W(R)/I_{1}$ as a subring, and its underlying additive group is isomorphic to $R^{\mathbb{N}}$ with addition given componentwise. Furthermore, if we write an element $(a_{0},a_{1},a_{2},\dotsc) \in R[[z]]^{F}$ as $a_{0}+a_{1}z+a_{2}z^{2}+\dotsb$, then multiplication in $R[[z]]^{F}$ is as follows:

\begin{multline}\label{eqnMultTwisted}
(a_{0}+a_{1}z+a_{2}z^{2}+\dotsc) \cdot (b_{0}+b_{1}z+b_{2}z^{2}+\dotsc) =\\= a_{0}b_{0} + (a_{0}^{p}b_{1}+a_{1}b_{0}^{p})z+(a_{0}^{p^{2}}b_{2}+a_{1}^{p}b_{1}^{p}+a_{2}b_{0}^{p^{2}})z^{2} + \dotsb.
\end{multline}

By the way this gives us a morphism of rings
\begin{equation}\label{eqnPowerSeriesTwisted}
\begin{split}
\mathcal{F} = 1\times F\times F^{2}\times\dotsb: R[[z]] &\to R[[z]]^{F}\\
a_{0}+a_{1}z+a_{2}z^{2}+\dotsb &\mapsto a_{0}+a_{1}^{p}z+a_{2}^{p^{2}}z^{2}+\dotsb,
\end{split}
\end{equation}
which is an isomorphism if $R$ is reduced and closed under taking $p$-th roots. Note furthermore, that $R[[z]]^{F}_{N}=R[[z]]^{F}/z^{N}R[[z]]^{F}$ is only true if $R$ is perfect or if $N=0$.




On $\Aff^{N}_{R}=\Spec R[x_{0},x_{1},\dotsc,x_{N-1}]$ the ring structure of $R[[z]]^{F}_{N}$ induces a structure of ring scheme, such that for any $R$-algebra $S$ we have $\Aff^{N}_{R}(S) = S[[z]]^{F}_{N}$. Similarly, we obtain a structure of $R[[z]]^{F}_{N}$-module scheme on $\Aff_{R}^{nN}=(\Aff_{R}^{N})^{n}$.
\smallskip

\textbf{A grading on on the coordinate ring of $\Aff_{R}^{nN}$}.
The crucial observation is the following: if we fix the grading
\begin{equation*}
\deg x_{i,j} = p^{j},\quad i=1,\dotsc,n; j=0,\dotsc,N-1
\end{equation*}
on the coordinate ring of $\Aff_{R}^{nN}$, then addition,
$$
a: \Aff_{R}^{nN}\times_{\Spec R}\Aff_{R}^{nN} \to \Aff_{R}^{nN},
$$
as well as multiplication by a scalar $\in R[[z]]^{F}_{N}$, are given by graded homomorphisms of the respective coordinate rings (see equation \eqref{eqnMultTwisted}).
\emph{From now on, we will consider the coordinate ring of $\Aff_{R}^{nN}$ endowed with this grading}. In the following section we study the set of graded ideals $I\subset R[x_{i,j}]$, where these morphisms factor, e.g. $\overline{a^{\#}}: R[x_{i,j}]/I \to R[x_{i,j}]/I\otimes_{R}R[x_{i,j}]/I$, and similarly for scalar multiplication.


\section{A moduli space for lattice schemes}\label{sectionLatticeSchemes}

\subsection{Multigraded Hilbert schemes}

Let us briefly summarize what we will need about multigraded Hilbert schemes. Our reference for this purpose is Haiman and Sturmfels \cite{haiman-2002}.

Let $S$ be a scheme and let $\Aff_{S}^{n}$ be the $n$-dimensional affine space over
$S$, given by the sheaf $\mathcal{A} = \ssh_{S}[x_{1},\dotsc,x_{n}]$ of polynomial rings. A grading of $\mathcal{A}$ by a semigroup $A$ is a semigroup homomorphism $\deg: \N^{n}\to A$. Identifying (over any open set $U\subset S$) the vector $u\in \N^{n}$ with the monomial $x_{1}^{u_{1}}\dotsb x_{n}^{u_{n}}\in \mathcal{A}$ induces a (multi-) grading
\begin{equation*}
 \mathcal{A} = \oplus_{a\in A} \mathcal{A}_{a},
\end{equation*}
where, over any open set $U\subset S$, $\mathcal{A}_{a}(U)$ is the $\ssh_{S}(U)$-span of the monomials of degree $a$.

A quasi-coherent sheaf of homogeneous ideals $\idsh \subset \mathcal{A}$ is called \emph{admissible} over $S$ (for short: an admissible ideal), if $(\mathcal{A}/\idsh)_{a}$ is a locally free sheaf of constant rank on $S$ for all $a\in A$.
With any admissible ideal $\idsh \subset \mathcal{A}$ we associate its \emph{Hilbert function}, given by
\begin{equation*}
 h_{\idsh}: A \to \N, a\mapsto \rk(\mathcal{A}/\idsh)_{a}.
\end{equation*}
Subschemes defined by admissible ideals will be called admissible as well, and by the Hilbert function of an admissible subscheme, we will mean the Hilbert function of its structure sheaf.

Let $h: A\to \N$ be any function supported on $\deg(\N^{n})$, and define the \emph{Hilbert functor} $\mathcal{H}_{S}^{h}$ from the category of $S$-schemes to the category of sets by
\begin{multline}
 \mathcal{H}_{S}^{h}(X) = \lbrace \text{admissible ideals } \idsh \subset \mathcal{A}\otimes_{\ssh_{S}}\ssh_{X}\\ \text{ such that } \rk(\mathcal{A}\otimes_{\ssh_{S}}\ssh_{X}/\idsh)_{a} = h(a) \text{ for all } a\in A \rbrace.
\end{multline}

This functor is representable by a scheme:

\begin{thm}[\cite{haiman-2002}]\label{thmGradedHilbertscheme}
There exists a quasiprojective scheme $H_{S}^{h}$ over $S$ which represents the functor $\mathcal{H}_{S}^{h}$. If the grading of $\mathcal{A}$ is positive, i. e. $1$ is the only monomial of degree $0$, then this scheme is even projective over $S$. \qed
\end{thm}

The scheme $H_{S}^{h}$ is called the \emph{multigraded Hilbert scheme}.
\subsection{Lattice schemes}\label{subsectionLatticeSchemes}

Recall that on $\Aff_{R}^{nN}$ we have a structure of $R[[z]]^{F}_{N}$-module scheme (section \ref{sectionPowerSeriesRings}). By $a: \Aff_{R}^{nN}\times_{\Spec R}\Aff_{R}^{nN} \to \Aff_{R}^{nN}$, resp. $m_{x}: \Aff_{R}^{nN}\to \Aff_{R}^{nN}$, denote the morphisms defining addition, resp. multiplication by the scalar $x\in R[[z]]^{F}_{N}$.

\begin{defn}\label{defLatticeScheme}
An admissible $R$-subscheme $V\subset \Aff_{R}^{nN}$ is a \emph{lattice scheme} over $R$, if the scheme-theoretic image under $a$ of $V\times_{R}V$, as well as the scheme-theoretic image under $m_{x}$ of $V$ (any $x\in R[[z]]^{F}_{N}$), factor through $V\subset \Aff_{R}^{nN}$. In other words, $V$ is required to be stable under the module operations on $\Aff_{R}^{nN}$.
\end{defn}

\begin{rem}
Let $I$ be the ideal defining an admissible $R$-subscheme $V\subset \Aff_{R}^{nN}$. Then the conditions of the definition can be restated by requiring that the morphisms
\begin{equation*}
\begin{split}
R[x_{i,j}]\xrightarrow{a^{\#}}R[x_{i,j}]&\otimes_{R}R[x_{i,j}] \to R[x_{i,j}]/I\otimes_{R}R[x_{i,j}]/I,\\
R[x_{i,j}]\xrightarrow{m_{x}^{\#}}&R[x_{i,j}] \to R[x_{i,j}]/I
\end{split}
\end{equation*}
factor through $R[x_{i,j}]/I$.
Note furthermore that any admissible $V$ contains the $0$-section and is trivially stable under taking additive inverses.
\end{rem}

Since we are considering admissible ideals, the target rings of the above morphisms are graded, with each graded component $R$-locally free of finite rank. Hence, for any of these morphisms, the image of the ideal $I$ defines a set of equations in $R$ defining the locus in $R$ where the respective morphism factors in the above manner. Thus, as a consequence of theorem \ref{thmGradedHilbertscheme}, we obtain

\begin{cor}
For any numerical function $h: \N \to \N$ there exists a projective $k$-scheme $\mathcal{M}^{(N),h}_{n}$ which parametrizes those admissible ideals of $\ssh_{\Aff_{k}^{nN}}$ that define closed $\Aff_{k}^{N}$-submodule schemes of $\Aff_{k}^{nN}$. More precisely: For any $k$-algebra $R$, the $R$-valued points of $\mathcal{M}^{(N),h}_{n}$ are exactly the sheaves of homogeneous ideals in $\ssh_{\Aff_{R}^{nN}}$ that are admissible over $R$ with Hilbert-function $h$ and define closed $\Aff_{R}^{N}$-submodule schemes of $\Aff^{nN}_{R}$. Equivalently, the $R$-valued points of $\mathcal{M}^{(N),h}_{n}$ are exactly the locally free closed $R$-subschemes of $\Aff_{R}^{nN}$ with Hilbert-function $h$, which are $\Aff_{R}^{N}$-submodule schemes of $\Aff^{nN}_{R}$.
\end{cor}

\begin{proof}
By theorem \ref{thmGradedHilbertscheme} we obtain a projective $k$-scheme $H_{k}^{h}$ parametrizing admissible ideals with Hilbert function $h$. The discussion above shows that the locus in $H_{k}^{h}$, where the universal family is stable under the module operations on $\Aff_{H_{k}^{h}}^{nN}$, is a closed subscheme of $H_{k}^{h}$. This is the scheme $\mathcal{M}^{(N),h}_{n}$.
\end{proof}

It is easy to see that the (set)-valued functor on the category of $k$-algebras
$$
\Loop^{F,\geq 0}G: R\mapsto G(R[[z]]^{F})
$$
is representable by an infinite dimensional affine group scheme over $k$. Similarly, the functor
$$
\Loop^{F,\geq 0}_{N}G: R\mapsto G(R[[z]]^{F}_{N})
$$
is representable by a (finite dimensional) algebraic group over $k$. We have $\Loop^{F,\geq 0}_{N}G = \operatorname{Res}_{k[[z]]^{F}_{N}/k}(G\times_{\Spec k}\Spec k[[z]]^{F}_{N})$, the Weil restriction of $G\times_{\Spec k}\Spec k[[z]]^{F}_{N}$. There are canonical morphisms of $k$-groups
\begin{equation}\label{eqnNonTwistedTwistedMorphisms}
\Loop^{\geq 0}G \to \Loop^{F,\geq 0}G, \quad \Loop^{F,\geq 0}G \to \Loop^{F,\geq 0}_{N}G,
\end{equation}
where the first of these morphisms is induced by the map \eqref{eqnPowerSeriesTwisted}. Moreover, there is an obvious algebraic operation of the $k$-group $\Loop^{F,\geq 0}G$ on $\Aff_{k}^{nN}$ given by multiplication of an $n\times n$-matrix and an $n\times 1$-vector over the ring of twisted truncated power series of length $N$. This operation factors through $\Loop^{F,\geq 0}_{N}G$ and respects the grading of the structure sheaf of $\Aff_{k}^{nN}$. Hence, by the universal property of $\mathcal{M}^{(N),h}_{n}$, we obtain an operation
$$
\Loop^{F,\geq 0}_{N}G \times_{\Spec k} \mathcal{M}^{(N),h}_{n} \to \mathcal{M}^{(N),h}_{n}.
$$
Certain orbits of this action and their closures will be the object of our interest:

\subsubsection{The standard lattice scheme associated to a dominant cocharacter}\label{subsubsectionStdLattice}

Let $\lambda\in \domcochar(T)$ and set $N:= \max \lbrace \langle \alpha, \lambda \rangle \suchThat \alpha \text{ a root of }T \rbrace$. Via the natural identification $\domcochar(T) \subset \mathbb{Z}^{n}$ let $\lambda = (\lambda_{1},\dotsc,\lambda_{n})$, and set $\tilde{\lambda} = \lambda - (\lambda_{n},\dotsc,\lambda_{n}) \in \N^{n}$. Then $N=\tilde{\lambda}_{1}$, and $\tilde{\lambda}$ defines a lattice scheme over $k$, given by the ideal $I(\lambda) \subset k[x_{i,j}; i=1,\dotsc,n; j=0,\dotsc,N-1]$, where
\begin{equation}\label{eqnIdeal}
I(\lambda) = (x_{1,0},\dotsc,x_{1,\tilde{\lambda}_{1}-1},\dotsc,x_{n-1,0},\dotsc,x_{n-1,\tilde{\lambda}_{n-1}-1}).
\end{equation}
Let us call this lattice scheme the \emph{standard lattice scheme} associated to $\lambda$. We denote its $\Loop^{F,\geq 0}G$-orbit in $\mathcal{M}^{(N),h}_{n}$ by $\mathcal{O}(\lambda)$, and its orbit-closure by $\mathcal{D}(\lambda)$. The latter will turn out to be closely related to a Demazure resolution of the Schubert variety $\mathcal{S}(\lambda)$ with respect to the `standard decomposition of $\lambda$` into minuscule dominant coweights, which we describe below.

\subsubsection{The standard decomposition of a dominant cocharacter}\label{subsubsectionStdDec}

For $\lambda\in \domcochar(T)$ choose $\mu_{i} =(1,\dotsc,1,0,\dotsc,0)$ such that $\lvert \mu_{i}\rvert := \text{number of entries in }\tilde{\lambda}\text{ which are}\geq i$, for $i=1,\dotsc,N$. This defines a decomposition of $\lambda$ into minuscule dominant coweights $\bar{\mu_{i}}$ as described in section \ref{sectionDemazureVarieties}. Obviously we have $\tilde{\lambda}=\mu_{1}+\dotsb+\mu_{N}$, and the $\mu_{i}$'s are ordered `by size`:
$$
\lvert \mu_{1} \rvert \geq \dotsb \geq \lvert \mu_{N} \rvert.
$$
We call this decomposition of $\lambda$ into minuscule dominant coweights the \emph{standard decomposition} of $\lambda$. For the rest of the paper we will assume $\lambda$, $\tilde{\lambda}$ and the $\mu_{i}$'s chosen in this way.


\section{Twisted linearity and flatness results}\label{sectionTechnical}

To give an idea of the relation of $\mathcal{D}(\lambda)$ to Demazure resolutions and thereby motivate the subsequent technical section, let's consider

\begin{ex}\label{exDemazure}
Choose $n=2$ (hence $G=\Sl_{2}$) and $\lambda=(1,-1)\in\domcochar(T)$. Then $N=2$, $\tilde{\lambda} = (2,0)$, and the standard lattice scheme associated to $\lambda$ is given by $I=I(\lambda)=\langle x_{1,0},x_{1,1}\rangle$. For convenience we rename the variables $x_{1,j} \mapsto x_{j}, x_{2,j} \mapsto y_{j}$, whence $I=\langle x_{0},x_{1}\rangle$.\\
\emph{Claim:} The $\Loop^{F,\geq 0}G$-orbit of $I$ consists of all ideals of the form
$$
\langle a_{0}x_{0}+b_{0}y_{0}, a_{0}^{p}x_{1} + a_{1}x_{0}^{p} + b_{0}^{p}y_{1} + b_{1}y_{0}^{p}\rangle,\text{ such that } a_{0} \neq 0\text{ or }b_{0} \neq 0. 
$$
To see this, one calculates the effect of a matrix $A=\begin{pmatrix} a_{0}+a_{1}z & b_{0}+b_{1}z \\ c_{0}+c_{1}z & d_{0}+d_{1}z\end{pmatrix}\in G(k[[z]]^{F})$ on the vector $x = \begin{pmatrix} x_{0}+x_{1}z \\ y_{0}+y_{1}z\end{pmatrix}\in (k[x_{i},y_{i}][[z]]^{F})^{2}$: We obtain
$$
A\cdot x \equiv \begin{pmatrix} (a_{0}x_{0}+b_{0}y_{0}) + (a_{0}^{p}x_{1} + a_{1}x_{0}^{p} + b_{0}^{p}y_{1} + b_{1}y_{0}^{p})z \\  (c_{0}x_{0}+d_{0}y_{0}) + (c_{0}^{p}x_{1} + c_{1}x_{0}^{p} + d_{0}^{p}y_{1} + d_{1}y_{0}^{p})z \end{pmatrix} \text{ modulo higher $z$-powers}.
$$
Since the operation of $\Loop^{F,\geq 0}G$ on $k[x_{i},y_{i}]$ is given by the transpose of this action, we see that the images of $x_{0}$ resp. $x_{1}$ are of the form $a_{0}x_{0}+b_{0}y_{0}$ resp. $a_{0}^{p}x_{1} + a_{1}x_{0}^{p} + b_{0}^{p}y_{1} + b_{1}y_{0}^{p}$, with $a_{0} \neq 0$ or $b_{0} \neq 0$. Thus the claim.\\
\emph{Claim:} the ideal $I(\lambda)$ is graded, with
$$I_{1}\cap \langle x_{0},y_{0}\rangle = \langle ax_{0}+by_{0}\rangle$$
and
$$I_{p}\cap \langle x_{0}^{p},y_{0}^{p},x_{1},y_{1}\rangle = \langle a^{p}x_{0}^{p}+b^{p}y_{0}^{p},cx_{0}^{p}+dy_{0}^{p}+a^{p}x_{1}+b^{p}y_{1}\rangle.$$
Only the latter equation requires an argument: we have to verify that, if $(ax_{0}+by_{0})\cdot P(x_{0},y_{0}) \in \langle x_{0}^{p},y_{0}^{p}\rangle$, then $P(x_{0},y_{0})=\gamma(ax_{0}+by_{0})^{p-1}$, for $\gamma\in k$. Assume $a\neq 0$ and consider the linear transformation of variables $x_{0}\mapsto (1/a)x_{0}-(b/a)y_{0}$, which stabilizes $\langle x_{0}^{p},y_{0}^{p}\rangle$. Hence $(ax_{0}+by_{0})\cdot P(x_{0},y_{0}) \in \langle x_{0}^{p},y_{0}^{p}\rangle$ if and only if $x_{0}\cdot P((1/a)x_{0}-(b/a)y_{0},y_{0}) \in \langle x_{0}^{p},y_{0}^{p}\rangle$. Thus we must have $P((1/a)x_{0}-(b/a)y_{0},y_{0})=\gamma'x_{0}^{p-1}$ for some $\gamma'\in k$, whence $P(x_{0},y_{0}) = \gamma(ax_{0}+by_{0})^{p-1}$. The case $b\neq 0$ is similar, which proves the claim.

Applying the absolute Frobenius morphism to the first of these modules yields, of course, a submodule of the second one, and identifying $x_{0}^{p},y_{0}^{p},x_{1},y_{1}$ with the standard basis of $k^{2\times 2}$, these define a descending sequence of lattices in $k^{2\times 2}$ in the sense of section \ref{sectionDemazureVarieties}:

\begin{multline*}
 \mathcal{L}_{0}=\Hom_{k}(k^{2\times 2},k)\supset \mathcal{L}_{1}=\Hom_{k}(k^{2\times 2}/\langle a^{p}x_{0}^{p}+b^{p}y_{0}^{p}\rangle,k) \supset \\ \supset \mathcal{L}_{2}= \Hom_{k}(k^{2\times 2}/\langle a^{p}x_{0}^{p}+b^{p}y_{0}^{p},cx_{0}^{p}+dy_{0}^{p}+a^{p}x_{1}+b^{p}y_{1}\rangle,k).
\end{multline*}

This corresponds to a point in the Demazure resolution of $\mathcal{S}(\lambda)$ (in the situation of this example there is only one Demazure resolution, since there is only one decomposition of $\lambda$ into minuscule dominant coweights).
\end{ex}

\begin{rem}
From the example it is clear that we will have to deal with certain Frobenius twists when relating points of $\mathcal{D}(\lambda)$ to points of a Demazure variety.
\end{rem}

As always, let $R$ be an arbitrary $k$-algebra, denote by $\mathfrak{X} = \lbrace x_{1},\dotsc,x_{m}\rbrace$ a set of indeterminates with $\deg x_{i}=p^{d_{i}}$ for any $i$, and set $$\fpv{l}{R} = \displaystyle\oplus_{d_{i}\leq l}Rx_{i}^{p^{l-d_{i}}}.$$
To carry out the construction of the example for any ($R$-valued) point in $\mathcal{D}(\lambda)$, possibly meeting the boundary of $\mathcal{O}(\lambda)$, we need to know that the corresponding ideals are well behaved, in a sense to be made precise, with respect to intersections with the $R$-submodules $\fpv{l}{R}$ of $R[x_{i,j}]$. To show this is the goal of this section.

The crucial property that the ideals which we consider will turn out to have, is subject of the following

\begin{defn}
Let $R$ be any $k$-algebra.
We call an element $f\in R[\mathfrak{X}]$ \emph{twisted-linear}, if $f \in \cup_{l} \fpv{l}{R}$.
We call an ideal $I\subset R[\mathfrak{X}]$ \emph{twisted-linear}, if $I$ is generated by a finite subset consisting of twisted-linear elements. Equivalently, a finitely generated ideal $I$ is twisted-linear, if and only if it is generated by $\cup_{l} (\fpv{l}{R}\cap I)$. Such an ideal is obviously graded.
\end{defn}

\begin{rem}
 A similar notion of `$q$-linearized polynomials`, where $q=p^{r}$, is considered by Elkies in \cite{elkies-1999}. There, a univariate polynomial $P$ is called $q$-linearized, if it has the form $P=\sum_{m=0}^{n}c_{m}T^{q^{m}}$. Our notion of twisted-linear ideal can be seen as the case of several such equations in several variables, where the equations are required to be homogeneous. However, the questions pursued by Elkies are of a different (Galois-theoretic) nature, whence I don't see further parallels in these papers.
\end{rem}

In the following, if $M$ is a graded $R$-module, we denote by $M_{d}$ its degree $d$-part. Note that $\fpv{N}{R}$ is a direct summand of $R[\mathfrak{X}]_{p^{N}}$. Note furthermore that, if $I$ is finitely generated, then $R[\mathfrak{X}]/I$ is finitely presented as an $R$-algebra, whence every homogeneous component of $R[\mathfrak{X}]/I$ is finitely presented as an $R$-module. Thus every such homogeneous component is $R$-flat if and only if it is $R$-projective if and only if it is locally free over $R$. In this case, the same properties hold for the homogeneous components of $I$.

\begin{lem}\label{lemIntersections}
Let $I\subset R[\mathfrak{X}]$ be a twisted-linear ideal such that $R[\mathfrak{X}]/I$ is $R$-flat. Assume that $I$ is generated by twisted-linear elements $f_{j}$ with $\deg f_{j}=p^{e_{j}}$ for each $j$. Then we have for any $N$
\begin{equation*}
\sum_{e_{j}\leq N} R f_{j}^{p^{N-e_{j}}} = \fpv{N}{R}\cap I.
\end{equation*}
\end{lem}

Before proving this lemma, we state two corollaries:

\begin{cor}\label{corIntersections}
Let $I\subset R[\mathfrak{X}]$ be a twisted-linear ideal such that $R[\mathfrak{X}]/I$ is flat over $R$. Then for any $R$-algebra $S$:
\begin{equation*}
(\fpv{N}{R}\cap I)\cdot S = \fpv{N}{S}\cap (I\otimes_{R} S)
\end{equation*}
(where the left expression denotes the image of $(\fpv{N}{R}\cap I)\otimes_{R} S$ in $S[\mathfrak{X}]$).
\end{cor}

\begin{proof}[Proof of corollary]
 The ideal $I\otimes_{R}S \subset S[\mathfrak{X}]$ satisfies the assumptions of the lemma, with $R$ replaced by $S$. In particular, it is generated by the images of the $f_{j}$ in $S[\mathfrak{X}]$. Hence we have
$$
\fpv{N}{S}\cap (I\otimes_{R} S) = \sum_{e_{j}\leq N} S f_{j}^{p^{N-e_{j}}} = (\fpv{N}{R}\cap I)\cdot S,
$$
where the first equality results from the lemma for the ideal $I\otimes_{R}S$, while the second is a consequence of the lemma for the ideal $I$ itself.
\end{proof}

\begin{cor}\label{corFlatness}
Let $I\subset R[\mathfrak{X}]$ be a twisted-linear ideal such that $R[\mathfrak{X}]/I$ is flat over $R$. Then the $R$-module $\fpv{N}{R}/(\fpv{N}{R} \cap I)$ is $R$-flat, as well as the quotient of degree $p^{N}$-components $R[\mathfrak{X}]_{p^{N}}/(\fpv{N}{R}+I_{p^{N}})$.
\end{cor}

\begin{proof}[Proof of corollary]
 The second module in question is the cokernel of the injection of $R$-modules
\begin{equation*}
 \fpv{N}{R}/(\fpv{N}{R}\cap I)\hookrightarrow (R[\mathfrak{X}]/I)_{p^{N}}.
\end{equation*}
Corollary \ref{corIntersections} states that this map is still injective after tensoring with any $R$-algebra $S$, so the long exact sequence of $\Tor$'s tells us that the cokernel is flat. And thus also $\fpv{N}{R}/(\fpv{N}{R}\cap I)$.
\end{proof}

Hence we see that we can in fact restate the equation in corollary \ref{corIntersections} in the form
\begin{equation*}
(\fpv{N}{R}\cap I)\otimes_{R} S = \fpv{N}{S}\cap (I\otimes_{R} S).
\end{equation*}
We interpret this equation as follows: \emph{cutting out of a twisted-linear ideal the twisted-linear part of given degree is functorial with respect to base change.}

\begin{proof}[Proof of lemma]
Since $R[\mathfrak{X}]_{1}/I_{1}$ is $R$-flat and of finite presentation, and since our claim is local on $R$, we can assume that $R[\mathfrak{X}]_{1}/I_{1}$ is free. Since invertible linear transformations of the degree 1-variables do not affect twisted-linearity nor flatness, we can assume without loss of generality that the $f_{j}$ of degree 1 are in fact in $\mathfrak{X}$, i.e. variables of degree 1. If all the $f_{j}$ have degree 1, we are done. Otherwise, we apply induction on $n$, where $p^{n}$ is the maximum degree of a generator $f_{j}$. By our assumption on $I_{1}$, we have $R[\mathfrak{X}]/I \simeq R[\mathfrak{X}']/I'$, where $\mathfrak{X}'$ is $\mathfrak{X}$ minus the variables generating $I_{1}$, and $I'$ is the image of $I$ under the quotient map $R[\mathfrak{X}]\to R[\mathfrak{X}']$. Hence, $I'$ is twisted-linear and generated by those $g_{j} = \operatorname{image}(f_{j}) \in R[\mathfrak{X}']$ which have degree $\geq p$.

Set $\mathfrak{X}'' = \lbrace x\in\mathfrak{X'}; \deg x>1\rbrace \cup \lbrace x^{p}\in \mathfrak{X}'; \deg x=1\rbrace$. Then the $g_{j}$ already lie in $R[\mathfrak{X}'']$ and generate an ideal $I''\subset R[\mathfrak{X}'']$. We obtain $R[\mathfrak{X}']/I' \simeq (R[\mathfrak{X}'']/I'')[p\text{-th roots of some variables in }\mathfrak{X}'']$. By faithful flatness of extension by $p$-th roots, $R[\mathfrak{X}'']/I''$ is flat over $R$. Now we only deal with variables of degree $\geq p$, and $I''$ is still twisted-linear, so for the moment we can think of all degrees divided by $p$ and apply the induction hypotheses, so to obtain the claim of the lemma for $I''\subset R[\mathfrak{X}'']$:
\begin{equation*}
\sum_{1 \leq e_{j}\leq N} R g_{j}^{p^{N-e_{j}}} = \fpv{N}{R}''\cap I''.
\end{equation*}
(Here, by abuse of notation, we denote by $g_{j}$ also the (unique) preimages of the $g_{j}$ in $R[\mathfrak{X}'']$. As for $R[\mathfrak{X}]$, we denote by $\fpv{N}{R}''$ and $\fpv{N}{R}'$ the modules of twisted-linear monomials of degree $p^{N}$ in the algebras $R[\mathfrak{X}'']$ and $R[\mathfrak{X}']$, respectively.)
Since
$$
R[\mathfrak{X}']/I' \simeq (R[\mathfrak{X}'']/I'')[p\text{-th roots}],
$$
we have $I'' = I'\cap R[\mathfrak{X}'']$, $\fpv{N}{R}' = \fpv{N}{R}'' \subset R[\mathfrak{X}'']$ and thus obtain
\begin{equation*}
\sum_{1 \leq e_{j}\leq N} R g_{j}^{p^{N-e_{j}}} = \fpv{N}{R}'\cap I'.
\end{equation*}
Now let $\varphi: R[\mathfrak{X}']\to R[\mathfrak{X}]$ be the canonical splitting of the quotient map $R[\mathfrak{X}]\to R[\mathfrak{X}']$. Then $I \subset \varphi(I')\oplus\sum_{m\neq 1}mR[\mathfrak{X}]$, where the sum runs over all monomials $m$ in variables in $\mathfrak{X}-\mathfrak{X}'$ not equal to 1. Since similarly $\fpv{N}{R} = \varphi(\fpv{N}{R}')\oplus \bigoplus_{x\in \mathfrak{X}-\mathfrak{X}'} Rx^{p^{N}}$, we have
$$
\fpv{N}{R} \cap I \subset \varphi(\fpv{N}{R}' \cap I')\oplus\bigoplus_{x\in \mathfrak{X}-\mathfrak{X}'} Rx^{p^{N}} = \sum_{e_{j}\leq N} R f_{j}^{p^{N-e_{j}}}.
$$
The opposite inclusion is trivial.
\end{proof}



By a similar induction argument we prove

\begin{lem}\label{lemCuttedFlatness}
Assume that the twisted-linear ideal $I\subset R[\mathfrak{X}]$ is generated in degrees $d\leq p^{n}$ and that the graded components $(R[\mathfrak{X}]/I)_{d}$ are flat over $R$ for $d\leq p^{n}$. Then $R[\mathfrak{X}]/I$ is $R$-flat.
\end{lem}

\begin{proof}
We proceed by induction on $n$. Since the graded components of $R[\mathfrak{X}]/I$ are of finite presentation, $(R[\mathfrak{X}]/I)_{d}$ is even locally free for $d\leq p^{n}$. Thus, since our claim is local on $R$, we can assume that $R[\mathfrak{X}]/\langle I_{1}\rangle$ is isomorphic to a polynomial ring $R[\mathfrak{X}']$, with $\mathfrak{X}'\subset \mathfrak{X}$. In case $n=0$ we are already done. Otherwise, the image of $I/\langle I_{1}\rangle$ in $R[\mathfrak{X}']$ is a twisted-linear ideal $I'\subset R[\mathfrak{X}']$, generated by the images $g_{j}$ of those generators of $I$ which have degree $\geq p$. Set $\mathfrak{X}'' = \lbrace x\in\mathfrak{X'}; \deg x>1\rbrace \cup \lbrace x^{p}\in \mathfrak{X}'; \deg x=1\rbrace$. Then the $g_{j}$ of degree $>1$ lie in $R[\mathfrak{X}'']$ and generate an ideal $I''\subset R[\mathfrak{X}'']$. We have
$$
R[\mathfrak{X}]/I \simeq R[\mathfrak{X}']/I' \simeq (R[\mathfrak{X}'']/I'')[p\text{-th roots of some variables in }\mathfrak{X}''].
$$
Still, $I''$ is twisted-linear and $R[\mathfrak{X}'']/I''$ is $R$-flat in degrees $\leq p^{n}$, since adjoining $p$-th roots is faithfully flat. In $\mathfrak{X}''$ there only occur variables of degrees $p^{e}$ with $e\geq 1$, so we can divide all degrees by $p$ and arrive at a situation where we can use the induction hypotheses: $R[\mathfrak{X}'']/I''$ is $R$-flat in \emph{all} degrees. But then so is $R[\mathfrak{X'}]/I' \simeq R[\mathfrak{X}]/I$.
\end{proof}

\begin{rem}\label{remCuttedFlatness}
\begin{enumerate}
\item For any ideal $I\subset R[\mathfrak{X}]$, we denote by $I^{\leq n}$ the ideal generated by the graded components of $I$ which have degree $\leq n$. Then, if $R[\mathfrak{X}]/I$ is $R$-flat and $I$ is twisted-linear, lemma \ref{lemCuttedFlatness} tells us that the same properties hold for $R[\mathfrak{X}]/I^{\leq p^{n}}$. In particular, we obtain from corollary \ref{corFlatness} that $R[\mathfrak{X}]_{p^{n}}/(\fpv{n}{R} + (I^{\leq p^{m}})_{p^{n}})$ is $R$-flat for all $m$ and $n$.
\item Let $I\subset R[\mathfrak{X}]$ be a twisted-linear ideal generated in degrees $\leq p^{n}$. For the quotient $R[\mathfrak{X}]/I$ to be $R$-flat it is sufficient that $(R[\mathfrak{X}]/I)_{p^{i}}$ is $R$-flat for $i=0,\dotsc,n$. This follows from lemma \ref{lemCuttedFlatness} by induction on $n$.
\end{enumerate}
\end{rem}

Combining these two remarks, we obtain a stronger version of the preceding lemma:

\begin{cor}\label{corStrongFlatness}
 Assume that the twisted-linear ideal $I\subset R[\mathfrak{X}]$ is generated in degrees $d\leq p^{n}$. Then $R[\mathfrak{X}]/I$ is $R$-flat if and only if the $R$-modules $\fpv{m}{R}/(\fpv{m}{R}\cap I)$ are flat for $m=0,\dotsc,n$.
\end{cor}

\begin{proof}
The `only if`-part was proved in corollary \ref{corFlatness}. For the `if`-part, we proceed by induction on $n$, the case $n=0$ being a consequence of lemma \ref{lemCuttedFlatness}, since $\fpv{0}{R}/(\fpv{0}{R}\cap I) = (R[\mathfrak{X}]/I)_{1}$. To verify the statement for $n>0$ assume $R[\mathfrak{X}]/I^{\leq p^{n-1}}$ is $R$-flat, and consider the exact sequence
$$
0 \to \fpv{n}{R}/(\fpv{n}{R}\cap I) \to R[\mathfrak{X}]_{p^{n}}/I_{p^{n}} \to R[\mathfrak{X}]_{p^{n}}/(\fpv{n}{R} + (I^{\leq p^{n-1}})_{p^{n}}) \to 0.
$$
Since the right-hand module is $R$-flat by corollary \ref{corFlatness}, and the left-hand module is $R$-flat by hypotheses, the same holds for the middle module. Now use (2) of remark \ref{remCuttedFlatness}.
\end{proof}

We are now prepared to show that twisted-linearity is a `closed condition` on the Hilbert scheme $H_{k}^{h}$. In order to give this statement a precise meaning and a neat formulation, let's sheafify the notion of twisted-linearity:

\begin{defn}
Let $X$ be any $k$-scheme, and let $\mathcal{I}\subset \mathcal{O}_{X}[\mathfrak{X}]$ be a sheaf of ideals. We say $\mathcal{I}$ is twisted-linear, if for every open affine subscheme $Y=\Spec R$ of $X$, $\Gamma(Y,\mathcal{I})\subset R[\mathfrak{X}]$ is a twisted-linear ideal.
\end{defn}

\begin{rem}\label{remTLOnAffines}
 Twisted linearity can be tested on an affine open covering: A sheaf of ideals $\mathcal{I}$ as in the previous definition is twisted-linear, if and only if there exists a covering by open affines, $X=\cup Y_{i}, Y_{i}=\Spec R_{i}$, such that every $I_{i}=\Gamma(Y_{i},\mathcal{I})\subset R_{i}[\mathfrak{X}]$ is a twisted-linear ideal. The `only if`-part is of course trivial. To verify the `if`-part, we can assume that $X=\Spec R$ is affine. Then $I=\Gamma(X,\mathcal{I})\subset R[\mathfrak{X}]$ is twisted-linear if and only if
$$
\langle I\cap \fpv{n}{R}, n\in \mathbb{N} \rangle = I.
$$
(That the property of being finitely generated, which we require for twisted-linear ideals, is local for the Zariski-topology on $\Spec R$, even for the faithfully flat topology, is a well known fact.)
By flatness of $R\to R_{i}$, this equation holds, if and only if it holds after localizing in $R_{i}$ for every $i$. But, using corollary \ref{corIntersections}, this is twisted-linearity on the $\Spec R_{i}$, which was our assumption.
\end{rem}

\begin{prop}\label{propTLClosed}
Let $p: X\to H_{k}^{h}$ be a morphism of $k$-schemes such that the corresponding sheaf of ideals on $X$ is twisted-linear. Assume that $p_{*}\mathcal{O}_{X}$ is a quasi-coherent $\mathcal{O}_{H^{h}_{k}}$-module and let $Y$ be the scheme-theoretic image of $p$. Then the corresponding sheaf of ideals $\mathcal{I}\subset \mathcal{O}_{Y}[\mathfrak{X}]$ is twisted-linear.
\end{prop}

\begin{proof}
The defining ideal-sheaf of $Y\subset H^{h}_{k}$ is equal to $\ker(\mathcal{O}_{H}\to p_{*}\mathcal{O}_{X})$. Hence to the map $f: X\to Y$ corresponds an \emph{injective} map of sheaves $\mathcal{O}_{Y}\to f_{*}\mathcal{O}_{X}$ and, in particular, on an affine open subset $\Spec R\subset Y$ we get $R\hookrightarrow \mathcal{O}_{X}(f^{-1}(\Spec R))$. Covering $f^{-1}(\Spec R)=\cup \Spec S_{i}$ by open affines, we obtain an injective map of rings
$$
R\hookrightarrow \mathcal{O}_{X}(f^{-1}(\Spec R)) \hookrightarrow \prod S_{i} =: S.
$$
By assumption, the ideal $J_{i}$ corresponding to $\Spec S_{i}\hookrightarrow f^{-1}(\Spec R) \to X \to H_{k}^{h}$ is twisted-linear. We have to verify that the same holds for the ideal $I$ corresponding to $\Spec R \hookrightarrow Y \to H_{k}^{h}$. Then the claim follows from the remark, as $Y$ can be covered by open affines like $\Spec R$.

We inductively construct a generating system for $I$ consisting of twisted-linear elements. First note that any element of degree 1 is trivially twisted-linear, i.e. $I^{\leq 1}$ is twisted linear. For the inductive step, assume that $I^{\leq p^{n-1}}$ is twisted linear. Then, by the remarks following lemma \ref{lemCuttedFlatness}, $M := R[\mathfrak{X}]_{p^{n}}/(\fpv{n}{R} + (I^{\leq p^{n-1}})_{p^{n}})$ is projective over $R$, whence, in the line below, the middle map is injective:
\begin{equation}\label{eqnItoMS}
I_{p^{n}} \to M \hookrightarrow M\otimes_{R}S \xrightarrow{\simeq} \prod (M\otimes_{R}S_{i}).
\end{equation}
To justify that the right hand map is an isomorphism, note that, if $M$ is even free, it is a finite product of copies of $R$, and the above map is indeed an isomorphism since arbitrary products commute. In general, $M$ is a direct summand (and hence also a direct factor) of a finitely generated free $R$-module $M\oplus M'$. This shows that $(M\otimes_{R}S)\oplus (M'\otimes_{R}S) \simeq (\prod M\otimes_{R} S_{i})\oplus (\prod M'\otimes_{R} S_{i})$.
Since the composition of the maps in \eqref{eqnItoMS} is zero ($I\otimes_{R}S_{i}=J_{i}$ being twisted-linear), so is the left-hand map. This shows that, in order to obtain a generating system of $I^{\leq p^{n}}$, we can extend a generating system of $I^{\leq p^{n-1}}$ by twisted-linear elements living in $\fpv{n}{R}$. This proves the inductive step, whence $I$ is twisted-linear.
\end{proof}

\begin{cor}\label{corClosure}
The sheaf of ideals corresponding to $\mathcal{D}(\lambda)\subset H_{k}^{h}$ is twisted-linear.
\end{cor}

\begin{proof}
The ideal $I(\lambda)$ (whose orbit is by definition $\mathcal{O}(\lambda)$) is twisted-linear. By definition of the ring operations in $R[[z]]^{F}$, the operation of $\Loop^{F,\geq 0}\Sl_{n}$ on $H_{k}^{h}$ preserves twisted-linearity, whence the sheaf of ideals corresponding to the inclusion $p: \mathcal{O}(\lambda)\subset H_{k}^{h}$ is twisted-linear. Now the claim follows from proposition \ref{propTLClosed}, since $\mathcal{D}(\lambda)$ is the closure (=scheme-theoretic image) of $\mathcal{O}(\lambda)$ in $H_{k}^{h}$.
\end{proof}


\section{The structure of $\mathcal{D}(\lambda)$}\label{sectionIsomorphismTheorem}
\subsection{An iterated bundle of ordinary Grassmannians}\label{sectionGrassmannBundle}

From now on, we denote $\mathfrak{X}=\lbrace x_{i,j}; i=1,\dotsc,n; j=0,\dotsc,N-1\rbrace$, i.e. $R[\mathfrak{X}]$ is the affine coordinate ring of $\Aff_{R}^{nN}$ with the grading $\deg x_{i,j}=p^{j}$, introduced in section \ref{sectionPowerSeriesRings}.

In the previous section we introduced the ideals $I^{\leq p^{i}}, i=0,1,\dotsc$, associated with a twisted linear admissible ideal $I\subset R[\mathfrak{X}]$. If we consider ideals $I$ with the property that $z^{\#}I^{\leq p^{i}}\subset I^{\leq p^{i-1}}$ (which is the case for $R$-valued points of $\mathcal{D}(\lambda)$), the assignments $I\to I^{\leq p^{i}}$ have a nice geometric interpretation, which we are going to describe in the sequel.


Denote by $h_{m}$ the Hilbert function of $k[\mathfrak{X}]/I(\lambda)^{\leq p^{m-1}}$ for $m=0,\dotsc,N$. In particular, $h_{0}$ is the Hilbert function of $k[\mathfrak{X}]$ (namely, $I(\lambda)$ contains no elements of degree $\leq p^{-1}$ whence $I^{\leq p^{-1}}=0$), and $h_{N}=h$, the Hilbert function of $k[\mathfrak{X}]/I(\lambda)$. For $m=0,\dotsc,N$ denote by $\mathcal{T}_{m}(\lambda)$ the following (set)-valued functor on the category of $k$-algebras:
\begin{multline}\label{eqnDefTm}
\mathcal{T}_{m}(\lambda)(R) = \lbrace \text{admissible twisted linear ideals }I\subset R[\mathfrak{X}]\\ \text{ with Hilbert function } h_{m} \text{ and such that } z^{\#}I^{\leq i} \subset I^{\leq i-1} \rbrace.
\end{multline}
Note that these functors are sheaves for the Zariski-topology by remark \ref{remTLOnAffines}. Note further that $\mathcal{T}_{0}(\lambda)=\Spec k$ and $\mathcal{D}(\lambda)\hookrightarrow \mathcal{T}_{N}(\lambda)$ as functors. Moreover, we have morphisms of functors
$$
\mathcal{T}_{m+1}(\lambda) \to \mathcal{T}_{m}(\lambda); \quad I\mapsto I^{\leq p^{m-1}},
$$
once we know that the Hilbert function of $R[\mathfrak{X}]/I$, for $I$ admissible, twisted linear and generated in degrees $\leq p^{m-1}$, is fully determined by its values on $\lbrace 0,\dotsc,p^{m-1} \rbrace$. But this follows from

\begin{prop}\label{propHF}
 The Hilbert function of an admissible twisted linear ideal $I\subset R[\mathfrak{X}]$, which is generated in degrees $\leq p^{m}$, determines the ranks of the $R$-modules $\fpv{0}{R}\cap I,\dotsc,\fpv{m}{R}\cap I$ and vice versa.
\end{prop}

\begin{proof}
 By corollary \ref{corFlatness} the $R$-modules $\fpv{i}{R}/\fpv{i}{R}\cap I$ are flat, whence the statement of the proposition makes sense.
Consider first ideals $I\subset k[\mathfrak{X}]$ having a \emph{monomial} twisted linear generating system that is, a generating system consisting of monomials $\xi_{1},\dotsc,\xi_{l}$, each a power of a variable $x_{i,j}$ and of degree $\leq p^{m}$. Hence, if the $\xi_{i}$ form a minimal generating system of $I$ (i.e. none of them can be omitted), they even form a regular sequence for $I$. Moreover, the number of elements in $\lbrace \xi_{i}\rbrace$ having a given degree is determined by the dimensions of the vector spaces $\fpv{i}{k}\cap I, i=0,\dotsc, m$. On the other hand Hilbert functions are additive with respect to short exact sequences as
$$
0 \to k[\mathfrak{X}]/\langle \xi_{1},\dotsc,\xi_{i-1}\rangle \xrightarrow{\cdot\xi_{i}} k[\mathfrak{X}]/\langle \xi_{1},\dotsc,\xi_{i-1}\rangle \to k[\mathfrak{X}]/\langle \xi_{1},\dotsc,\xi_{i}\rangle \to 0.
$$
Thus the Hilbert function of $R[\mathfrak{X}]/I$ is determined by the dimensions of $\fpv{i}{k}\cap I, i=0,\dotsc, m$, and vice versa, as claimed. Now we show that the general case reduces to the case just studied: By flatness of $R[\mathfrak{X}]/I$ and by corollaries \ref{corFlatness} and \ref{corIntersections} we may assume $R=k$. Consider the action on $H^{h'}_{k}$ (where $h'$ denotes the Hilbert function of $I$) of the 1-parameter subgroup
 $$\lbrace D(t)=\diag(t^{-n+1},t^{-n+3},\dotsc,t^{n-1}); t\in k\rbrace \subset T\subset \Sl_{n}(k).$$
By properness of $H^{h'}_{k}$ the orbit of $I$ extends to a closed curve $C\subset H_{k}^{h'}$, and by proposition \ref{propTLClosed} the ideal $I'$ corresponding to $t=0$ is generated by twisted linear homogeneous elements of degrees $\leq p^{m}$. Let $\sum_{i=1}^{n}P_{i}(x_{i,0},\dotsc,x_{i,N-1})$ be such an element, the $P_{i}(x_{i,0},\dotsc,x_{i,N-1})$ denoting polynomials of the same degree $d$ over $k$ in $N$ variables, or the zero polynomial. Then, by construction, the multiplicative group acts (via $D(t)$) on $P_{i}$ with weight $(-n-1+2i)d$. Since $I'$ is a fixed point under the action of $D(t)$, this shows that each of the summands $P_{i}(x_{i,0},\dotsc,x_{i,N-1})$ is itself contained in $I'$. But from this it follows easily that $I'$ is even a monomial ideal generated by elements of the form $x_{i,j}^{p^{e_{i,j}}}$ for some non-negative integers $e_{i,j}$. Namely, if $P(x_{i,0},\dotsc,x_{i,N-1}) = a_{e}x_{i,0}^{p^{d+e}}+\dotsb+a_{0}x_{i,e}^{p^{d}}$ with $a_{0}\neq 0$, then we have $x_{i,j}^{p^{d+e-j}}\in I'$ for every $0\leq j\leq e$, since $I'$ is stable under the map $x_{i,j}\mapsto x_{i,j-1}^{p}$ (`comultiplication by $z$`).
Of course, the Hilbert functions of $k[\mathfrak{X}]/I$ and $k[\mathfrak{X}]/I'$ coincide, and the application of corollary \ref{corIntersections} to the specializations $\Spec k\to C$ given by $t=0$ and $t=1$, respectively, shows that also the dimensions of $\fpv{i}{k}\cap I$ and $\fpv{i}{k}\cap I'$ coincide for $i=0,\dotsc,m$. Thus we are indeed reduced to the special case of monomial ideals.
\end{proof}

For the following recall the standard decomposition of $\lambda$ defined in section \ref{subsubsectionStdDec}.

\begin{thm}\label{thmGrassmannBundle}
For each $m=0,\dotsc,N$ the morphism $\mathcal{T}_{m}(\lambda)\to \mathcal{T}_{m-1}(\lambda)$ is relatively representable by a bundle of ordinary Grassmannians $\Grass_{\lvert \mu_{m}\rvert,n}$. In particular, it is smooth of relative dimension $\lvert \mu_{m} \rvert (n-\lvert \mu_{m} \rvert)$.
\end{thm}

\begin{cor}\label{corDT}
 The functor $\mathcal{T}_{N}(\lambda)$ is representable by a smooth, connected $k$-scheme of dimension $\sum_{m=1}^{N}\lvert \mu_{m} \rvert (n-\lvert \mu_{m} \rvert) = \dim \tilde{\Sigma}(\mu_{1},\dotsc,\mu_{N})$. The functorial map $\mathcal{D}(\lambda)\hookrightarrow \mathcal{T}_{N}(\lambda)$ is thus a closed immersion of $k$-schemes.
\end{cor}

\begin{proof}[Proof of corollary]
 Recall that $\mathcal{T}_{0}(\lambda)=\Spec k$ and use induction on $N$.
\end{proof}

\begin{proof}[Proof of theorem]
Let $I$ be any $R$-valued point of $\mathcal{T}_{m-1}(\lambda)$ and consider the exact sequence of flat $R$-modules
\begin{multline*}
0 \to K := \ker(z^{\#}) \to \fpv{m-1}{R}/(\fpv{m-1}{R}\cap I) \xrightarrow{z^{\#}}\\ \to \fpv{m-1}{R}/(\fpv{m-1}{R}\cap I) \to \sum_{i=1}^{n}Rx_{i,m-1} \to 0.
\end{multline*}
Then $K$ is flat over $R$, and hence locally free since it is of finite presentation, and of rank $n$. By proposition \ref{propHF} and corollary \ref{corStrongFlatness} the fiber of $\mathcal{T}_{m}(\lambda)\to \mathcal{T}_{m-1}(\lambda)$ is in functorial bijection with the set of $R$-submodules $L\subset K$ such that $K/L$ is flat over $R$ and of constant rank
\begin{equation}\label{eqnRankMu}
\rk (I(\lambda)^{\leq p^{m-1}}/I(\lambda)^{\leq p^{m-2}})_{p^{m-1}} = \rk I(\lambda)_{p^{m-1}}-\rk I(\lambda)_{p^{m-2}} = \lvert \mu_{m}\rvert.
\end{equation}
But these are exactly the $R$-valued points of a $\Grass_{\lvert \mu_{m}\rvert,n}$-bundle over $\Spec R$. In particular, the relative dimension of $\mathcal{T}_{m}(\lambda)\to \mathcal{T}_{m-1}(\lambda)$ is equal to $(\lvert \mu_{m}\rvert)(n-\lvert \mu_{m}\rvert)$.
\end{proof}

\subsection{The relation with Demazure varieties}

Recall the ideal $I(\lambda)$ of the standard lattice scheme associated to $\lambda$ and its orbit closure $\mathcal{D}(\lambda)$.
Denote by $h:\N\to \N$ the Hilbert function of $I(\lambda)$. Furthermore, let $N$ be as in \ref{subsubsectionStdLattice}, i.e. $N = \max \lbrace \langle \alpha, \lambda \rangle \suchThat \alpha \text{ a root of }T \rbrace$, and, as in the previous section, let $\mathfrak{X}=\lbrace x_{i,j}; i=1,\dotsc,n; j=0,\dotsc,N-1 \rbrace$.

Denote by $\Frob_{R}: R[\mathfrak{X}]\to R[\mathfrak{X}]$ the relative Frobenius morphism over $R$.
Let $F^{*}$ be the pullback-functor on the category of $R$-modules along the absolute Frobenius morphism $R\to R$, whose effect on a submodule of $R[\mathfrak{X}]$ is raising coefficients of its elements to the $p$-th power. Note that applying $\Frob_{R}$ to a submodule of $R[\mathfrak{X}]$ and then pulling the image back via $F^{*}$ yields the image of the original submodule under the \emph{absolute} Frobenius.
The pullback functor $F^{*}$ induces an $\mathbb{F}_{p}$-morphism $\varphi: \G_{\Gl_{n}}^{(N)}\to\G_{\Gl_{n}}^{(N)}; \mathcal{L}\mapsto F^{*}(\mathcal{L})$. We set
$$
\Phi := \varphi^{N-1}\times\dotsb\times\varphi\times \id: \prod_{i=1}^{N}\G_{\Gl_{n}}^{(N)}\to\prod_{i=1}^{N}\G_{\Gl_{n}}^{(N)}.
$$

Furthermore set $M := \fpv{N-1}{R}$ and note that on $M$ we have an endomorphism $z^{\#}$ which is defined by restriction of scalar multiplication with $z$ on $R[\mathfrak{X}]$: it sends $x_{i,j}^{p^{N-1-j}}$ to $x_{i,j-1}^{p^{N-j}}$ if $j>0$ and to $0$ otherwise. Let $\lbrace e_{i,j}\rbrace$ be the dual basis of $\lbrace x_{i,j}^{p^{N-1-j}} \rbrace$ and consider the dual map of $z^{\#}$, which we again denote by $z$:
$$
z: V := \Hom_{R}(M,R)\to \Hom_{R}(M,R)\simeq V; e_{i,j-1}\mapsto e_{i,j}.
$$
Note that $V \simeq R^{nN}$ and its $R$-submodules with projective quotient which are stable under $z$ are lattices in the sense of section \ref{sectionDemazureVarieties} (they correspond to precisely those lattices $\mathcal{L}\subset R[[z]]^{n}$ which contain $z^{N}R[[z]]^{n}$).


\begin{thm}\label{thmRelFunctors}
Let $\Sigma(\lambda) = \tilde{\Sigma}(\mu_{1},\dotsc,\mu_{N})$ denote the variety of lattice chains defined in section \ref{sectionDemazureVarieties}, with $\mu_{1},\dotsc,\mu_{N}$ the standard decomposition of $\lambda$. Then there is a closed immersion $\iota: \mathcal{D}(\lambda) \to \prod_{i=1}^{N}\G_{\Gl_{n}}^{(N)}$, such that the following diagram commutes:

\begin{xy}
 \xymatrix{
	\mathcal{D}(\lambda) \ar[r]^{\iota}\ar[d]^{\sigma} & \prod_{i=1}^{N}\G_{\Gl_{n}}^{(N)} \ar[d]^{\Phi} \\
	\Sigma(\lambda) \ar@^{(->}[r] & \prod_{i=1}^{N}\G_{\Gl_{n}}^{(N)}.
}
\end{xy}

Moreover, the map $\sigma$ is a universal homeomorphism and the map from $\mathcal{D}(\lambda)$ to the fiber product of $\Sigma(\lambda)$ with $\prod_{i=1}^{N}\G_{\Gl_{n}}^{(N)}$ is a nil-immersion.
\end{thm}

\begin{proof}
Recall that every $R$-valued point of $\mathcal{D}(\lambda)$ `is` a twisted-linear ideal $I$ by corollary \ref{corClosure}. For $l=1,\dotsc,N$ we set $L_{l} := \Frob_{R}^{N-l}(\fpv{l-1}{R}\cap I)$ and $\mathcal{L}_{l} := \Hom_{R}(M/L_{l}, R)$ (Compare example \ref{exDemazure} at the beginning of the previous section!).
Then by corollary \ref{corFlatness}, the modules $L_{l}$ are projective. By lemma \ref{lemIntersections} this assignment is functorial, and finally, since any twisted-linear ideal $I$ is generated by the sets $\fpv{l}{R}\cap I$, we have a functorial \emph{injection}. Since split exact sequences are preserved by both $\Frob_{R}$ and $\Hom_{R}(-,R)$, $V/\mathcal{L}_{l}$ is again projective, whence the map $\iota: I\mapsto (\mathcal{L}_{1},\dotsc,\mathcal{L}_{N})$ is well-defined. Note furthermore the equation $F^{*}L_{l-1} = \Frob_{R}^{N-l}\Frob(\fpv{l-2}{R}\cap I) \subset L_{l}$. It implies
\begin{equation}
F^{*}\mathcal{L}_{l-1} = F^{*}\Hom_{R}(M/L_{l-1}, R) = \Hom_{R}(M/F^{*}L_{l-1}, R) \supset \mathcal{L}_{l},
\end{equation}
which proves that the image of $\Phi\circ\iota$ indeed consists of descending lattice chains. Again by projectivity, the rank of successive quotients $F^{*}\mathcal{L}_{l-1}/\mathcal{L}_{l}$ is constant on $\Spec R$, i.e. equal to $\lvert \mu_{l}\rvert$ by equation \eqref{eqnRankMu}. Thus the map $\sigma$.

It remains to check that the immersion $\alpha: \mathcal{D}(\lambda)\to S := \Sigma(\lambda)\times_{\prod_{i=1}^{N}\G_{Gl_{n}}^{(N)}}\prod_{i=1}^{N}\G_{Gl_{n}}^{(N)}$ of $k$-varieties is indeed a nil-immersion, i.e. that it is surjective. (Since $S\to \Sigma(\lambda)$ is a universal homeomorphism, this will imply that $\sigma$ is a universal homeomorphism as well.)

Observe that the map $\Phi\circ\iota$ (and hence $\sigma$) is $\Loop^{\geq 0}\Sl_{n}$-equivariant for the following $\Loop^{\geq 0}\Sl_{n}$-actions: on $\mathcal{D}(\lambda)$ consider the morphism $\Loop^{\geq 0}\Sl_{n}\to \Loop^{F,\geq 0}\Sl_{n}$ (see \eqref{eqnNonTwistedTwistedMorphisms}) composed with the canonical $\Loop^{F,\geq 0}\Sl_{n}$-action, and on $\Sigma(\lambda)$ consider the morphism $\Loop^{\geq 0}\Sl_{n}\times_{k}k \to \Loop^{\geq 0}\Sl_{n}$ given by $k\to k;x\mapsto x^{p^{N-1}}$, composed with the natural action of $\Loop^{\geq 0}\Sl_{n}$. This shows that the image of $\mathcal{O}(\lambda)$ is dense in $\Sigma(\lambda)$, and in particular both have the same dimension. By finiteness of $\Phi$ we see that $\alpha$ has dense image as well, and is therefore surjective.
\end{proof}


Let us recall the iterated bundle of Grassmannians $\mathcal{T}_{N}(\lambda)$ defined in section \ref{sectionGrassmannBundle}. In corollary \ref{corDT} we saw that there is a natural closed immersion $\mathcal{D}(\lambda)\hookrightarrow \mathcal{T}_{N}(\lambda)$, where the dimension of the latter equals the dimension of $\Sigma(\lambda)$. But in theorem \ref{thmRelFunctors} we have now seen that this is also the dimension of $\mathcal{D}(\lambda)$. Thus we have proved

\begin{cor}\label{corMain}
The varieties $\mathcal{D}(\lambda)$ and $\mathcal{T}_{N}(\lambda)$ are equal. In particular, $\mathcal{D}(\lambda)$ is an iterated bundle of ordinary Grassmannians.
\end{cor}

\begin{ex}\label{exPullback}
We illustrate the theorem by calculating explicitly the smallest nontrivial example: Let $n=2$ and $N=2$. We are thus dealing with ideals in the polynomial ring $R[x_{0},x_{1},y_{0},y_{1}]$ for some $k$-algebra $R$. Let $\lambda=(1,-1)\in \domcochar(T)$ whence $\tilde{\lambda}=(2,0)$, $\mu_{1}=\mu_{2}=(1,0)$ and $I(\lambda)=(x_{0},x_{1})$.

\begin{prop}
 In the above situation the variety $\mathcal{D}(\lambda)$ is isomorphic to the projective space bundle $X=\Proj_{\mathbb{P}_{k}^{1}}(\ssh_{\mathbb{P}_{k}^{1}}\oplus\ssh_{\mathbb{P}_{k}^{1}}(-2p))$. The boundary of the big cell is the divisor $\mathbb{P}^{1}_{k}\times_{k}\lbrace\infty\rbrace$.
\end{prop}
Note that $X$ can be explicitly constructed by gluing two copies of $U=\Aff_{k}^{1}\times\mathbb{P}^{1}_{k}$ via the self-inverse isomorphism
$$
\chi: (\Aff_{k}^{1}-\lbrace 0\rbrace)\times\mathbb{P}^{1}_{k}\xrightarrow{\simeq}(\Aff_{k}^{1}-\lbrace 0\rbrace)\times\mathbb{P}^{1}_{k}; (a,(c:d))\mapsto (1/a,(c:a^{2p}d)).
$$
We will use this description in the

\begin{proof}[Proof of the proposition]
 Define two maps on $R$-valued points.
\begin{equation*}
 \begin{split}
  \varphi: U(R) &\to \mathcal{D}(\lambda)(R)\\
  (a,(c:d)) &\mapsto (ax_{0}+y_{0}, cx_{0}^{p}+da^{p}x_{1}+dy_{1})\\
  \psi: U(R) &\to \mathcal{D}(\lambda)(R)\\
  (a,(c:d)) &\mapsto (x_{0}+ay_{0}, -cy_{0}^{p}+da^{p}y_{1}+dx_{1})
 \end{split}
\end{equation*}
The maps $\varphi(R)$ and $\psi(R)$ are clearly functorial in $R$, so they constitute morphisms of $k$-schemes.
Furthermore, observe that $\varphi\circ\chi = \psi$ on $(\Aff_{k}^{1}-\lbrace 0\rbrace)\times\mathbb{P}^{1}_{k}$. This shows that $\varphi$ and $\psi$ give a morphism $f: X \to \mathcal{D}(\lambda)$.
To show that this is an isomorphism, we have to find an inverse $f(R)^{-1}$ which is functorial in $R$. But by proposition \ref{corClosure} every $R$-valued point of $\mathcal{D}(\lambda)$ is twisted-linear, i.e. locally of the form $(ax_{0}+by_{0}, cx_{0}^{p}+da^{p}x_{1}+db^{p}y_{1})$ or $(ax_{0}+by_{0}, -cy_{0}^{p}+da^{p}x_{1}+db^{p}y_{1})$ (depending on whether $b\neq 0$ or $a\neq 0$) with $a$ or $b$ a unit in $R$ and $c$ or $d$ a unit in $R$. Such an ideal defines in a functorial way an $R$-valued point of $X$.
Obviously, the divisor $\lbrace d=0\rbrace = \mathbb{P}^{1}_{k}\times_{k}\lbrace\infty\rbrace$ maps bijectively to the boundary of the big cell: it parametrizes the ideals of the form $(ax_{0}+y_{0}, x_{0}^{p})$ or $(x_{0}+ay_{0}, y_{0}^{p})$.
\end{proof}

Similarly, we compute the Demazure-variety $\Sigma(\lambda)$. Its $k$-valued points are descending chains of subspaces $k^{2\times 2}_{x_{0},x_{1},y_{0},y_{1}}=\mathcal{L}_{0}\supset \mathcal{L}_{1}\supset \mathcal{L}_{2}$ which are stable under multiplication by $z$: $x_{0}\mapsto x_{1},y_{0}\mapsto y_{1}$ and of codimension $1$ and $2$, respectively.
Hence one sees like in the proposition above that $\Sigma(\lambda)\simeq \Proj_{\mathbb{P}_{k}^{1}}(\ssh_{\mathbb{P}_{k}^{1}}\oplus\ssh_{\mathbb{P}_{k}^{1}}(-2))$, using two charts on $R$-valued points
\begin{equation*}
 \begin{split}
  \varphi': U(R) &\mapsto \mathcal{D}_{\mu}(R)\\
  (a,(c:d)) &\mapsto (ax_{0}+y_{0}, cx_{0}+dax_{1}+dy_{1})\\
  \psi': U(R) &\mapsto \mathcal{D}_{\mu}(R)\\
  (a,(c:d)) &\mapsto (x_{0}+ay_{0}, -cy_{0}+day_{1}+dx_{1}),
 \end{split}
\end{equation*}
and gluing them via $(a,(c:d))\mapsto (1/a,(c:a^{2}d))$.

Since the Frobenius morphism commutes with dualization, to compute the map $\sigma: \mathcal{D}(\lambda)\to \Sigma(\lambda)$ of theorem \ref{thmRelFunctors} explicitly, we may compute in terms of defining relations of lattices, instead of generators of lattices. Hence, using the charts from above, we immediately obtain the picture

\begin{center}
 \begin{xy}
 \xymatrix{
	\Proj_{\mathbb{P}_{k}^{1}}(\ssh_{\mathbb{P}_{k}^{1}}\oplus\ssh_{\mathbb{P}_{k}^{1}}(-2p)) \ar[r]^{\sigma}\ar[d] & \Proj_{\mathbb{P}_{k}^{1}}(\ssh_{\mathbb{P}_{k}^{1}}\oplus\ssh_{\mathbb{P}_{k}^{1}}(-2)) \ar[d] \\ \mathbb{P}_{k}^{1} \ar[r]^{Frob}\ar[d] & \mathbb{P}_{k}^{1} \ar[d] \\
        \Spec k \ar[r]^{Frob} & \Spec k,
}
\end{xy}
\end{center}

where the map $\sigma$ is given on the respective charts by
\begin{equation*}
  (a,(c:d)) \mapsto (a^{p},(c:d)).
\end{equation*}
\end{ex}

It is elementary to check that the upper square is cartesian, by looking at the transition functions defining the respective line bundles. This is in agreement with theorem \ref{thmRelFunctors}: namely, a diagram relating this with the square of theorem \ref{thmRelFunctors} has the form

\bigskip

\begin{xy}
 \xymatrix{
	\mathcal{D}(\lambda)\ar@/^ 0.5cm/[rrr] \ar[r]^{\iota}\ar[d]^{\sigma} & \G_{\Gl_{2}}^{(2)}\times\G_{\Gl_{2}}^{(2)} \ar[r]^{\pi_{1}} \ar[d]^{\varphi\times\id} & \G_{\Gl_{2}}^{(2)} \ar[d]^{\varphi} & \mathbb{P}^{1}  \ar[l]\ar[d]^{\Frob} \\
	\Sigma(\lambda)\ar@/_ 0.5cm/[rrr] \ar@^{(->}[r] & \G_{\Gl_{2}}^{(2)}\times\G_{\Gl_{2}}^{(2)} \ar[r]^{\pi_{1}} & \G_{\Gl_{2}}^{(2)} & \ar[l] \mathbb{P}^{1}.
}
\end{xy}

\bigskip

Here the horizontal maps on the right are given by identifying $\mathbb{P}^{1}_{k}$ with the space of lattices $\mathcal{L}_{1}$ with $z\mathcal{L}_{0} \subset \mathcal{L}_{1} \subset \mathcal{L}_{0}$. Note that the middle square is trivially cartesian, while the right hand square is not. Thus also the left hand square fails to be cartesian and the map $\alpha$ of theorem \ref{thmRelFunctors} is indeed a nontrivial nil-immersion.

\subsection{Demazure varieties as schemes of lattices with infinitesimal structure.}

The $k$-valued points of $\mathcal{D}(\lambda)$ are lattice schemes in $\Aff_{k}^{nN}$ (see definition \ref{defLatticeScheme}), which, in general, will not be reduced. For instance, in the situation of example \ref{exPullback}, the set of points in $\mathcal{D}(\lambda) \simeq \Proj_{\mathbb{P}_{k}^{1}}(\ssh_{\mathbb{P}_{k}^{1}}\oplus\ssh_{\mathbb{P}_{k}^{1}}(-2p))$ which have non-reduced fibers in the universal family, is exactly the divisor $d=0$ (the boundary of the big cell $\O(\lambda)$). The corresponding lattice schemes have the form $(ax_{0}+by_{0}, x_{0}^{p}, y_{0}^{p})$ with $a$ and $b$ not both zero. In general, we obtain

\begin{cor}\label{corInfStruct}
 If $k$ is a perfect field of characteristic $p$, then the map $\sigma: \mathcal{D}(\lambda)\to \Sigma(\lambda)$ defines a bijection of $k$-valued points. In particular, for any lattice $\mathcal{L}\in \mathcal{S}(\lambda)$, the fiber of $\mathcal{L}$ in $\Sigma(\lambda)$ can be interpreted as a variety of infinitesimal structures on $\mathcal{L}$. The lattices having non-trivial infinitesimal structure are exactly those lying in the boundary of the big cell $\O(\lambda)$.
\end{cor}

\begin{proof}
 Only the very last assertion requires a proof: by definition, the lattices in $\mathcal{O}(\lambda)$ are reduced (since $I(\lambda)$ is). On the other hand, take any lattice $\mathcal{L}$ in the boundary of $\mathcal{O}(\lambda)$. Such a lattice maps to the boundary of $\mathcal{S}(\lambda)$ under $\mathcal{D}(\lambda)\to\Sigma(\lambda)\to \mathcal{S}(\lambda)$ by the $\Loop^{\geq 0}\Sl_{n}$-equivariance of $\sigma$ (see proof of theorem \ref{thmRelFunctors}). Thus its reduced structure corresponds to a point in $\mathcal{O}(\lambda')$ for some $\lambda' < \lambda$ (Bruhat-order), whence it has a Hilbert function different from that of $I(\lambda)$. Thus, by constancy of the Hilbert function on $\mathcal{D}(\lambda)$, the reduced structure of $\mathcal{L}$ cannot belong to $\mathcal{D}(\lambda)$.
\end{proof}

Finally, let us briefly study the invariants of lattices in $\mathcal{M}^{(N),h}_{n}$. (if $d$ denotes `the` vector of elementary divisors of a lattice $\mathcal{L}\subset k((z))^{n}$, then by the \emph{invariants} of $\mathcal{L}$ we mean the vector $(\operatorname{val}_{z}d_{1},\dotsc,\operatorname{val}_{z}d_{n})$ with entries ordered by decreasing size.) We will need the following purely combinatorial lemma about partitions of a given positive integer.

\begin{lem}\label{lemPartitions}
Let $\sigma,\sigma'\in\mathbb{N}^{n}$ with $\lvert\sigma\rvert = \lvert\sigma'\rvert$ and such that $\sigma_{i}\geq \sigma_{i+1}$ and $\sigma'_{i}\geq \sigma'_{i+1}$ for all $1\leq i\leq n$. Let $N$ be the maximum of all the $\sigma_{i}$ and $\sigma'_{i}$ and set
\begin{equation*}
\begin{split}
\tau_{i} := \text{number of entries in }\sigma\text{ which are}\geq i,\\
\tau'_{i} := \text{number of entries in }\sigma'\text{ which are}\geq i,
\end{split}
\end{equation*}
for $i=1,\dotsc,N$ (the `dual partitions` for $\sigma$ and $\sigma'$).
Then, with respect to the Bruhat order, $\sigma \geq \sigma'$ if and only if $\tau \leq \tau'$.
\end{lem}

\begin{proof}
By definition, $\tau \leq \tau'$ if and only if for all $1\leq i\leq N$, $0\leq \sum_{j=1}^{i}(\tau'_{j}-\tau_{j})$. Since $\lvert \tau \rvert = \lvert \tau' \rvert$, this is equivalent to $0\leq \sum_{j=i}^{N}(\tau_{j}-\tau'_{j}) = \sum_{j=i}^{N}(\#\lbrace l\suchThat \sigma_{l}\geq j\rbrace - \#\lbrace l\suchThat \sigma'_{l} \geq j\rbrace) = \sum_{l=1}^{n}(\max\lbrace \sigma_{l},i\rbrace - \max\lbrace \sigma'_{l},i\rbrace)$. We show that this holds if $\sigma \geq \sigma'$: Let $1\leq r\leq n$ be the smallest index such that $\sigma_{r} \leq i$ (case 1) or $\sigma'_{r} \leq i$ (case 2) (of course these cases don't exclude each other), and let $1\leq s \leq n$ be the smallest index such that both $\sigma_{s}$ and $\sigma'_{s}$ are $\leq i$. Then in case 1 we have
$$
\sum_{j=i}^{N}(\tau_{j}-\tau'_{j}) = \sum_{l=1}^{r}(\sigma_{l}-\sigma'_{l}) + \sum_{l=r+1}^{s}(i-\sigma'_{l}) \geq \sum_{l=1}^{s}(\sigma_{l}-\sigma'_{l}),
$$
while case 2 yields
$$
\sum_{j=i}^{N}(\tau_{j}-\tau'_{j}) = \sum_{l=1}^{r}(\sigma_{l}-\sigma'_{l}) + \sum_{l=r+1}^{s}(\sigma_{l}-i) \geq \sum_{l=1}^{r}(\sigma_{l}-\sigma'_{l}).
$$
Both expressions are thus non-negative if $\sigma \geq \sigma'$, which proves one direction. The other implication holds by duality: $\tau$ (resp. $\tau'$) can be regarded as the dual partition for $\sigma$ (resp. $\sigma'$). Thus all the arguments remain valid if we interchange $\sigma$ and $\tau'$ (resp. $\sigma'$ and $\tau$).
\end{proof}

Let $h: \mathbb{N}\to \mathbb{N}$ be the Hilbert function of $I(\lambda)$ and let $V \subset \Aff_{k}^{nN}$ be a lattice scheme corresponding to a $k$-valued point in $\mathcal{M}^{(N),h}_{n}$.
Since $k$ is assumed to be perfect, the map in \eqref{eqnPowerSeriesTwisted} for $\bar{R}=k$ is bijective, whence $k[[z]]^{F}$ is by definition isomorphic to $k[[z]]$. Via this isomorphism we can regard the set $V(k)$ of $k$-valued points as a lattice $z^{N}k[[z]]^{n}\subset \mathcal{L} \subset z^{-N}k[[z]]^{n}$ (after multiplication by a suitable power of $z$). From lemma \ref{lemPartitions} we obtain

\begin{cor}
Assume that the defining ideal of $V$ is twisted-linear, and let $\lambda'$ denote the invariants (ordered by decreasing size) of the corresponding lattice $\mathcal{L} \subset k((z))^{n}$. Then $\lambda' \leq \lambda$.
\end{cor}

\begin{proof}
Let $I(V)$ denote the defining ideal of $V$.
Multiplying with a suitable matrix in $\Sl_{n}(k[[z]]^{F})$ we may assume that $V_{red}$ is given by the ideal $I(V)_{red}=\langle x_{1,0},\dotsc,x_{1,\tilde{\lambda}'_{1}-1},\dotsc,x_{n,0},\dotsc,x_{n,\tilde{\lambda}'_{n}-1}\rangle$ for some $\tilde{\lambda'}\in \mathbb{N}^{n}$ with $\tilde{\lambda}'_{i}\geq \tilde{\lambda}'_{i+1}$. Since $I(V)$ is twisted linear, its Hilbert function $h$ determines the dimension of the $k$-vector space $I(V)\cap \fpv{N}{k} = I(V)_{red}\cap\fpv{N}{k}$: it is therefore equal to the dimension of $I(\lambda)\cap \fpv{N}{k}$, whence the determinants of the respective lattices conincide. In other words: $\lvert\tilde{\lambda}\rvert = \lvert\tilde{\lambda'}\rvert$. Now we have to show that $\tilde{\lambda} \geq \tilde{\lambda'}$, or equivalently (by lemma \ref{lemPartitions}) $\mu \leq \mu'$, where $\mu$ and $\mu'$ are the respective duals in the sense of lemma \ref{lemPartitions}. Since $\mu_{i} = \dim_{k} I(V)_{p^{i-1}}\cap \fpv{i-1}{k}$, while $\mu_{i}' = \dim_{k} (I(V)_{red})_{p^{i-1}}\cap \fpv{i-1}{k}$, the claim follows.
\end{proof}

Thus the $k$-valued points in $\mathcal{M}_{n}^{(N),h}$ which are given by twisted linear ideals correspond to lattices with invariants $\leq \lambda$. (I don't know, if every $k$-valued valued point in $\mathcal{M}^{(N),h}_{n}$ is twisted-linear.) Thinking of the analogous situation for Schubert varieties in the affine Grassmannian (where the Schubert variety $\mathcal{S}(\lambda)$ parametrizes exactly those lattices with invariants $\leq \lambda$) one could be tempted to think that $\mathcal{M}_{n}^{(N),h}$ and $\mathcal{D}(\lambda)$ coincide. However, looking once more at the simple situation of example \ref{exPullback}, we see that this is in general not the case:

Let $P=k[x_{0},x_{1},y_{0},y_{1}]$, $\lambda=(2,0)$ and set $I=(y_{0},x_{0}^{p})\in \mathcal{T}:=\mathcal{T}_{2}$, $\mathcal{M}:=\mathcal{M}_{2}^{(1),h}$. Note that a deformation $\tilde{I}\in \mathcal{M}(k[\epsilon])$ of $I$ is given by 2 generators with indeterminate coefficients $a,b,c\in k$:
\begin{equation*}
 \begin{split}
  &g_{1} := y_{0}+\epsilon ax_{0},\\
  &g_{2} := x_{0}^{p}+\epsilon (by_{1}+cx_{1}).
 \end{split}
\end{equation*}
Hence $\dim \T_{I}\mathcal{M} = 3$, while the dimension of $\T_{I}\mathcal{T}$ is $2$, e.g. by theorem \ref{thmGrassmannBundle}. (In a more elementary way, we could argue that on $\T_{I}\mathcal{T}$ we have the additional condition that $z^{\#}g_{2} \subset (g_{1})$, which means that $\epsilon(by_{0}^{p}+cx_{0}^{p})$ must be a multiple of $g_{1}^{p}=y_{0}^{p}$. This forces $c=0$, whence again $\dim \T_{I}\mathcal{T} = 2$).

Of course, this exhibits only a difference in the infinitesimal structures at the point $I(\lambda)$. But also the topological spaces of $\mathcal{M}_{n}^{(N),h}$ and $\mathcal{D}(\lambda)$ differ in general:
Let $n=4, \lambda=(1,1,-1,-1)$, whence $N=2$ and $\tilde{\lambda}=(2,2,0,0)$. Then $I(\lambda)=\langle x_{0},x_{1},y_{0},y_{1} \rangle \subset k[x_{i},y_{i},z_{i},w_{i}; i=0,1]$. On the other hand, consider the twisted-linear ideal $I=\langle x_{0},y_{0},z_{0}^{p},z_{1}\rangle$ in the same polynomial ring. Certainly, it has the same Hilbert function as $I(\lambda)$, and it defines a lattice scheme with invariants $(1,0,0,-1)$. Thus it is a $k$-valued point of $\mathcal{M}^{(N),h}_{n}$. However, it violates the condition $z^{\#}I = z^{\#}I^{p} \subset I^{\leq 1}$ which, by \eqref{eqnDefTm}, is satisfied by points of $\mathcal{T}_{2}(\lambda)=\mathcal{D}(\lambda)$.

\bibliographystyle{amsalpha}
\bibliography{DemazureInfStruct}
\end{document}